\documentclass[12pt,reqno]{amsart}
\usepackage{amsmath,amsfonts,amsthm,amssymb,color}
\usepackage[usenames,dvipsnames]{xcolor}
\usepackage[T1]{fontenc}
\usepackage{enumerate}
\usepackage{empheq}
\usepackage{hyperref}
\hypersetup{
     colorlinks   = true,
     citecolor    = blue,
     linkcolor    = blue
}
\usepackage[textsize=small]{todonotes}
\usepackage{pdfsync}{\tiny }

\usepackage[left=1in, right=1in, top=1.1in,bottom=1.1in]{geometry}
\newtheorem{theorem}{Theorem}[section]
\newtheorem{corollary}[theorem]{Corollary}

\newtheorem{remark}[theorem]{Remark}

\numberwithin{equation}{section}

\newtheorem{thm}{Theorem}[section]
\newtheorem{lemma}[thm]{Lemma}

\newtheorem{definition}[thm]{Definition}

\def\1{{\rm l}\hskip -0.21truecm 1}

\def\EE{\mathbb{E}}

\def \eref#1{\hbox{(\ref{#1})}}

\begin{document}

\title[An implicit numerical scheme for a class of BDSDEs]{An implicit numerical scheme for a class of backward doubly stochastic differential equations}
 
\date{}
\author{Yaozhong Hu\and
 David Nualart \and
 Xiaoming Song }

\address{Yaozhong Hu and David Nualart: Department of Mathematics, University of Kansas,   Lawrence, KS 66045, USA.\quad\rm  yhu@ku.edu,nualart@ku.edu}

\thanks{D. Nualart was supported by the NSF grant  DMS1512891.}

\address{Xiaoming Song: Department of Mathematics, Drexel University, Philadelphia, PA 19104, USA. \quad \rm song@math.drexel.edu}

\begin{abstract}
In this paper, we consider a class of backward doubly stochastic differential equations (BDSDE  for short) with general terminal value and general random generator. Those BDSDEs do not involve any forward diffusion processes.   By using the techniques of Malliavin calculus, we are able to establish the $L^p$-H\"{o}lder continuity of the solution pair. Then, an implicit numerical scheme for the BDSDE is proposed and the rate of convergence is obtained in the $L^p$-sense.
As a by-product,  we obtain an explicit representation of the process $Y$ in the solution pair to a linear BDSDE with random coefficients. 
\end{abstract}

  \maketitle

\section{Introduction} \setcounter{equation}{0}
Let $\{W_t\}_{0\leq t\leq T}$ and $\{B_t\}_{0\leq t\leq T}$ be two independent standard Brownian motions on a probability space $(\Omega,\mathcal{F},\mathbb{P})$. Let $\mathcal{N}$ denote the class of $\mathbb{P}$-null sets. For each $t\in[0,T]$, we define
\[
\mathcal{F}_t=\mathcal{F}_t^W\vee\mathcal{F}_{t,T}^B,
\]
where
\[
\mathcal{F}_t^W=\sigma\{W_s, 0\leq s\leq t\}\vee \mathcal{N}\ \mbox{and}\ \mathcal{F}_{t,T}^B=\sigma\{B_s-B_t, t\leq s\leq T\}\vee \mathcal{N}.
\]
The purpose of this paper is to study an implicit numerical scheme for the
 following backward doubly stochastic
differential equation (BDSDE for short) 
\begin{equation}  \label{bdsde}
Y_t=\xi+\int_t^Tf(r, Y_r,Z_r)dr+\int_t^Tg(Y_r)d\overleftarrow{B_r}-\int_t^TZ_rdW_r,\quad 0\le t\le T\,,
\end{equation}
where $\xi$ is a  given terminal value, $f$ is a  given (random)
generator, $g$ is a deterministic function,  and $\int_t^Tg(Y_r)d\overleftarrow{B_r}$ denotes the backward It\^{o} integral.   BDSDEs were introduced by Pardoux and Peng in \cite{PP94} as a generalization of the classical backward stochastic differential equations (BSDEs for short)
considered   in the pioneering paper  \cite{PP90} by these authors,
in order to give a probabilistic
representation of solutions to a class of systems of quasilinear stochastic partial differential equations.

There is an extensive literature on numerical   schemes for BSDEs. Most of the works deal with the case where the terminal random variable $\xi$ is a functional of a forward diffusion process $\{X_t\}_{0\le t\le T}$  and the generator $f$ is of the form $f(t,X_t, Y_t, Z_t)$, where $f$ is a deterministic function. Starting from the four-step numerical scheme considered by Ma, Protter and Yong  in  \cite{MPY}, many authors have contributed to this problem
(see, for instance, \cite{Ba,BD,BT,BDM,GLW,IDR,MPSS}).  In  \cite{Zh}, Zhang introduced a discretization method based on the $L^2$-regularity of the process $Z$.  In \cite{HNS10} the present authors  considered  the case of a BSDE with a general terminal value   $\xi$ which is
twice differentiable in the sense of Malliavin calculus and the
first and second  Malliavin derivatives satisfy some integrability conditions
and we also made similar assumptions for the generator $f$.
  In this general  framework,  we were
able to  obtain  an estimate of the form $\mathbb{E} |Z_t-Z_s|^p\le K|t-s|^{\frac p2}$ for any $p\ge 2$ and we applied this result to study the rate of convergence of different types of numerical schemes, including an implicit one.

Unlike the case of BSDEs, numerical schemes for BDSDEs have received much less attention. The presence of a backward  It\^o stochastic integral creates  additional difficulties when deriving the  path regularity of the process $Z$ and computing the rate of convergence of numerical schemes. 
In the present paper, we consider an implicit numerical scheme introduced  by Bachouch, Ben Lasmar, Matoussi and Mnif in a unpublished note \cite{BBMM}. Under the general assumptions on the terminal variable $\xi$ and the generator $f$ considered in \cite{HNS10} we have been able to show the H\"older continuity of the process $Z$ and to derive a rate of convergence of the scheme (see the estimate  \eqref{s-3-34-1}) in Theorem \ref{t-3-10}. The approach is similar to that developed in \cite{HNS10}, however, there is a new significant difficulty.   Unlike BSDE, linear BDSDEs do not have an explicit solution in exponential form and the desired representation of the process $Y$ (see formula \eqref{e-3-5}) cannot be  deduced  directly from It\^o's formula. We shall get around this difficulty by using  Taylor expansion with   some explicit computations.   
On the other hand,  we need to assume that  the function $g$ in Equation \eqref{bdsde} is a   deterministic  functional  of the process $Y$.
 
The paper is organized as follows. Section  2 contains some preliminaries on backward stochastic integrals and Malliavin calculus. The Malliavin calculus provides a representation of the random variable $Z_t$ as the derivative $D_tY_t$, which is very useful to derive the H\"older continuity and other regularity properties of $Z$. Our main results are stated in Section 3 and Sections 4 to 7 are devoted to the proofs.

The results of this paper still hold if the Brownian motions are multidimensional, but we have restricted the presentation to the one-dimensional case for the sake of simplicity.

\section{ Notation and preliminaries}
\subsection{Backward and forward It\^{o} integrals}

Recall all the notations defined at the beginning of the previous section and for any $t\in [0,T]$ define
\[
\mathcal{G}_t=\mathcal{F}_t^W\vee\mathcal{F}^B_{0,T}.
\]
Note that $\{\mathcal{F}_t\}_{0\leq t\leq T}$ is not a filtration, while $\{\mathcal{G}_t\}_{0\leq t\leq T}$ is a filtration. 

We say that a stochastic process $\{u_t\}_{0\leq t\leq T}$ is $\mathcal{G}$-adapted ($\mathcal{F}$-adapted, respectively),  if  $u_t$ is $\mathcal{G}_t$-measurable ($\mathcal{F}_t$-measurable, respectively) for all $t\in[0,T]$. Consider the following spaces of random variables and processes:
\begin{itemize}
\item $M^{p}$, for any $p\geq 2$, denotes the class of $L^p$-integrable random variables $F$
with a stochastic integral representation of the form%
\begin{equation*}
F=\mathbb{E}\left(F|\mathcal{G}_0\right)+\int_{0}^{T}u_{t}dW_{t},
\end{equation*}%
where $u$ is a $\mathcal{G}$-adapted stochastic process satisfying $\sup_{0\leq t\leq T}\mathbb{E}%
|u_{t}|^{p}<\infty $;
 \item $H_{\mathcal{F}}^{p}([0,T])$ ($H^p_{\mathcal{G}}([0,T])$, respectively), for any $p\geq 1$, denotes the set of 
 jointly measurable and $\mathcal{F}$-adapted ($\mathcal{G}$-adapted, respectively)  processes $\{\varphi_t\}_{0\leq t\leq T} $ satisfying
\[
\Vert \varphi \Vert _{H^{p}}=\left( \mathbb{E}\left( \int_{0}^{T}|\varphi
_{t}|^{2}dt\right) ^{\frac{p}{2}}\right) ^{\frac{1}{p}}<\infty;
\]
\item $S_{\mathcal{F}}^{p}([0,T])$ ($S^p_{\mathcal{G}}([0,T])$, respectively), for any $p\geq 1$, denotes the set of  all RCLL (right-continuous with left limits) $\mathcal{F}$-adapted ($\mathcal{G}$-adapted, respectively) processes
$\{\varphi_t\}_{0\leq t\leq T} $ satisfying 
 $$\Vert \varphi \Vert _{S^{p}}=\left(\mathbb{E} \sup_{0\leq t\leq
T}|\varphi _{t}|^{p} \right)^{\frac{1}{p}}<\infty.$$

\end{itemize}

The backward It\^{o} integral is  similar to the classical (forward) It\^{o} integral,  if we just reverse the time. Therefore, It\^{o}'s formula and It\^{o}'s isometry also hold for the backward It\^{o} integral. In particular,  It\^{o}'s formula has the following form because of the backward integral (see Lemma 1.3 in \cite{PP94}).
\begin{lemma}
Suppose that $\beta$, $\gamma$ and $\sigma$ are processes  in $ H^2_{\mathcal{F}}([0,T])$. Let the $\mathcal{F}$-adapted process $\alpha$ have the following form
\[
\alpha_t=\alpha_0+\int_0^t\beta_sds+\int_0^t\gamma_sd\overleftarrow{B_s}+\int_0^t\sigma_sdW_s,\,\quad   0\leq t\leq T.
\]
Then,
\[
\alpha_t^2=\alpha_0^2+2\int_0^t\alpha_s\beta_sds+2\int_0^t\alpha_s\gamma_sd\overleftarrow{B_s}+2\int_0^t\alpha_s\sigma_sdW_s-\int_0^t\gamma_s^2ds+\int_0^t\sigma_s^2ds.
\]
More generally, if $f\in C^2(\mathbb{R})$, then we have the following It\^{o}'s formula
\begin{eqnarray*}
f(\alpha_t)&=&f(\alpha_0)+\int_0^tf'(\alpha_s)\beta_sds+\int_0^tf'(\alpha_s)\gamma_sd\overleftarrow{B_s}+\int_0^tf'(\alpha_s)\sigma_sdW_s\\
&& -\frac{1}{2}\int_0^tf''(\alpha_s)\gamma_s^2ds+\frac{1}{2}\int_0^tf''(\alpha_s)\sigma_s^2ds,
\end{eqnarray*}
for $t\in [0,T]$.
\end{lemma}

\subsection{Malliavin calculus with respect to the Brownian motion $W$}
In this subsection,   we present  some preliminaries on Malliavin calculus and we
refer  the reader to the books   \cite{hu16} and  \cite{N06}  for
more details.

Let $\mathbf{H}=L^2([0,T])$ be the separable Hilbert space of all
square integrable real-valued functions on the interval $[0,T]$
with scalar product denoted by $\langle
\cdot,\cdot\rangle_\mathbf{H}$. The norm of an element $h\in
\mathbf{H}$ will be denoted by $\Vert h\Vert_\mathbf{H}$. For any
$h\in \mathbf{H}$ we put $W(h)=\int_0^Th(t)dW_t$ and $B(h)=\int_0^Th(t)d{B}_t$.

For any $m,\, n\in\mathbb{N}$, we denote by $C_p^\infty(\mathbb{R}^{m+n})$ the set of all infinitely
 differentiable functions $g: \mathbb{R}^{m+n}\rightarrow
\mathbb{R}$ such that $g$ and all of its partial derivatives have
polynomial growth. We    make use of the notation $\partial_i g=\frac{\partial
g}{\partial x_i}$ whenever $g\in C^1(\mathbb{R}^{m+n})$.

Let $\mathcal{S}$ denote the class of smooth and cylindrical random variables such
that a random variable $F\in \mathcal {S}$ has the form
\begin{equation}\label{smooth}
F=g(W(h_1),\dots,W(h_m),B(k_1),\dots,B(k_n)),
\end{equation}
where $g$ belongs to $C_p^\infty(\mathbb{R}^{m+n})$, $h_1,\dots,h_m$ and $k_1, \dots, k_n$
are in $\mathbf{H}$, and $m, n\in\mathbb{N}$.

For a smooth and cylindrical random variable $F$  of the
form (\ref{smooth}), its Malliavin derivative with respect to $W$  is the $\mathbf{H}$-valued random variable
given by
\[
D_tF=\sum_{i=1}^m\partial_i g(W(h_1),\dots,W(h_m),B(k_1),\dots, B(k_n))h_i(t), \, t\in[0,T].
\]
For any $p\geq1$ we will denote the domain of $D$ in $L^p(\Omega)$
by $\mathbb{D}^{1,p}$, meaning that $\mathbb{D}^{1,p}$ is the
closure of the class of smooth and cylindrical  random variables $\mathcal{S}$ with
respect to the norm
\[
\Vert F\Vert_{1,p}=\left(\EE|F|^p+\EE\Vert
DF\Vert_\mathbf{H}^p\right)^{\frac{1}{p}}.
\]
We can define the iteration of the operator $D$ in such a way that
for a smooth  and cylindrical random variable $F$, the iterated derivative $D^kF$
is a random variable with values in $\mathbf{H}^{\otimes k}$. 
For every $p\geq1$ and any natural number $k\geq1$ we introduce
the seminorm on $\mathcal{S}$ defined by
\[
\Vert F\Vert_{k,p}=\left(\EE|F|^p+\sum_{j=1}^k\EE\Vert
D^jF\Vert_{\mathbf{H}^{\otimes j}}^p\right)^{\frac{1}{p}}.
\]
We will denote by $\mathbb{D}^{k,p}$ the completion of the family
of smooth  and cylindrical random variables $\mathcal{S}$ with respect to the norm
$\Vert\cdot\Vert_{k,p}$.

Let $\mu$ be the Lebesgue measure on $[0,T]$. For any $k\geq 1$
and $F\in\mathbb{D}^{k,p}$, the derivative
\[
D^kF=\{D^k_{t_1,\dots,t_k}F,\,t_i\in[0,T],\, i=1,\dots,k\},
\]
is a measurable function on the product space
$ [0,T]^k \times \Omega$, which is defined a.e. with respect to the
measure $  \mu^k \times \mathbb{P} $.

 We denote by $\mathbb{L}_{a}^{1,p}$  the set of
real-valued jointly measurable processes $u=\left\{
u_t\right\}_{0\leq t\leq T} $ such that

\begin{description}
\item[(i)] For each $t\in[0,T]$, $u_t$ is $\mathcal{F}_t$-measurable.
\item[(ii)] For almost all $t\in \lbrack 0,T],\,u_t\in
\mathbb{D}^{1,p}$.

\item[(iii)] $\mathbb{E}\left(
\left(\int_{0}^{T}|u_t|^{2}dt\right)^{\frac{p}{2}}+\left(\int_{0}^{T}\int_{0}^{T}|D_{\theta
}u_t|^{2}d\theta dt\right)^{\frac{p}{2}}\right) <\infty $.
\end{description}

\setcounter{equation}{0}
\section{Main results}

In this section, we will give a summary of main results whose proofs will be provided in subsequent  sections.

\subsection{Estimates on the solutions of BDSDEs}
We assume that the generator  in the BDSDE (\ref{bdsde}) is a jointly measurable function $f: \mathbb{R}_+ \times \mathbb{R}^2 \times \Omega  \rightarrow \mathbb{R}$, such that,
for each fixed pair $(y,z)\in\mathbb{R}^2$,  $f(t, y,z)$   is $\mathcal{F}_t$-measurable for all $t\in[0,T]$. We suppose also that the terminal value  $\xi$ is an  $\mathcal{F}_T$-measurable
random variable.
\begin{definition}
A solution to the BDSDE (\ref{bdsde}) is a pair of $\mathcal{F}$-adapted processes $(Y,Z)$ such that: $\int_0^T|Z_t|^2dt<\infty$,  $\int_0^T\left|g(Y_s)\right|^2ds<\infty$, $\int_0^T|f(t,Y_t,Z_t)|dt<\infty$, a.s., and
\[
Y_t=\xi+\int_t^Tf(r, Y_r,Z_r)dr+\int_t^Tg(Y_s)d\overleftarrow{B_s}-\int_t^TZ_rdW_r,\quad 0\le t\le T.
\]
\end{definition}

The following estimate on the solution to the BDSDE \eqref{bdsde} will play an important role in obtaining the path regularity of $Z$.

\begin{theorem}\label{T.2.1} 
Fix $q\geq2$. Suppose that $\mathbb{E}|\xi|^q<\infty$ and the generator $f$ satisfies $f(\cdot, 0,0) \in H^q_{\mathcal{F}}([0,T])$. We also assume that the generator $f$ and the function $g$ satisfy the following 
Lipschitz conditions: there exists a positive constant $L$ such that
\begin{eqnarray*}
|f(t,y_{1},z_{1})-f(t,y_{2},z_{2})|&\leq&
L(|y_{1}-y_{2}|+|z_{1}-z_{2}|)\,,\ \mbox{a.s.}\  \mu \times \mathbb{P}\ \mbox{on}\  [0,T] \times \Omega ,\\
|g(y_1)-g(y_2)|&\leq &L|y_1-y_2|,
\end{eqnarray*}%
for all $y_1,y_2\in\mathbb{R}$ and $z_1,z_2\in\mathbb{R}$.
 Then,
there exists a unique solution pair $(Y,Z)\in
S_{\mathcal{F}}^{q}([0,T])\times H_{\mathcal{F}}^{q}([0,T])$ to
Equation (\ref{bdsde}). Moreover, we have the following estimate
for the solution
\begin{equation}
\mathbb{E}\sup_{0\leq t\leq T}|Y_{t}|^{q}+\mathbb{E}\left(
\int_{0}^{T}|Z_{t}|^{2}dt\right) ^{\frac{q}{2}}\leq K\left(
\mathbb{E}|\xi
|^{q}+\mathbb{E}\left( \int_{0}^{T}|f(t,0,0)|^{2}dt\right) ^{\frac{q}{2}%
}+ |g(0)|^q\right) ,  \label{e.2.1}
\end{equation}%
where $K$ is a constant depending only on $L$, $q$ and $T$.\newline
\end{theorem}

\begin{corollary}\label{l.3.7}
Under the assumptions in Theorem \ref{T.2.1}, let $(Y,Z)\in
S^q_\mathcal{F}([0,T])\times H_\mathcal{F}^q([0,T])$ be the unique
solution pair to the BDSDE (\ref{bdsde}). If\, $\sup_{0\le t\le
T}\mathbb{E}|Z_t|^{q}<\infty$, then there exists a constant $C$, depending on $L$, $q$, $T$ and the quantity appearing in the right-hand side of \eqref{e.2.1},
such that, for any $s,\,t\in[0,T]$,
\begin{equation}\label{Hy}
\mathbb{E}|Y_t-Y_s|^q\le C|t-s|^{\frac{q}{2}}.
\end{equation}
\end{corollary}

\subsection{Linear BDSDEs}

As we will see later, the component $Z$ of the solution of a given BDSDE can be represented  in terms of the Malliavin derivative of the solution $Y$, which
 satisfies  a linear BDSDE with random coefficients. In order to describe the
properties of $Z$  we first study a class of linear BDSDEs.

We consider the following linear BDSDE:
\begin{equation}\label{lbdsde}
Y_t=\xi+\int_t^T(\alpha_sY_s+\beta_sZ_s+f_s)ds+\int_t^T\gamma_sY_sd\overleftarrow{B_s}-\int_t^TZ_sdW_s,\ 0\leq t\leq T,
\end{equation}
where the processes $\alpha,  \beta, \gamma$ and $f$  are  jointly measurable and $\mathcal{F}$-adapted. 

We impose the following boundedness condition on the coefficients.
\begin{description}
\item[ (H1)]   The processes $\{{\alpha }_{t}\}_{0\le t\le T}$, $\{{\beta }%
_{t}\}_{0\le t\le T}$  and $\{\gamma_t\}_{0\leq t\leq T}$ are uniformly bounded, namely, there exists a constant $L>0$ such that
\[
|\alpha_t|+|\beta_t|+|\gamma_t|\leq L,\ a.s. \  \mu\times\mathbb{P} \ \mbox{on}\ [0,T]\times \Omega.
\]
\end{description}

Under the condition \textbf{(H1)}, we define a process $\rho$ by
\begin{equation}\label{joint1}
\rho_t=\exp\left\{\int_0^t \beta_sdW_s+\int_0^t\gamma_sd\overleftarrow{B_s}+\int_0^t\left(\alpha_s-\frac{1}{2}\beta_s^2-\frac{1}{2}\gamma_s^2\right)ds\right\}, \quad 0\leq t\leq T.
\end{equation}
Note that $\rho_t$ is not $\mathcal{F}_t$-measurable but $\mathcal{G}_t$-measurable and that $\rho$ has a continuous version. In fact, $\rho$ does not
satisfy any stochastic differential equation and hence It\^{o}'s formula cannot be applied to the process $\rho$. However, we are still able to prove the following result.
\begin{theorem}\label{p-3-3}
Let   $\xi\in L^2(\Omega)$ and $f\in H_{\mathcal{F}}^2([0,T])$. Assume that the processes $\alpha$, $\beta$ and $\gamma$ satisfy condition \textbf{(H1)}. Let  $\rho$ be defined in \eqref{joint1}. Then, there exists a unique solution $(Y,Z)\in S_{\mathcal{F}}^2([0,T])\times H^2_{\mathcal{F}}([0,T])$ to \eqref{lbdsde} and the following equation holds
\begin{equation}\label{e-3-5}
Y_t\rho_t=\xi\rho_T+\int_t^T\rho_sf_sds-\int_t^T(\rho_sZ_s+Y_s\rho_s\beta_s)dW_s,\quad 0\leq t\leq T.
\end{equation}
\end{theorem}
 
 As a consequence,  we have the following representation for $Y$:
\begin{equation}\label{e-3-6}
Y_t=\rho_t^{-1}\mathbb{E}\left(\left.\xi\rho_T+\int_t^T\rho_sf_sds \right|\mathcal{G}_t\right).
\end{equation}

The following result on the moment estimate of the increment of $Y$ in the linear BDSDE \eqref{lbdsde} will play a critical role in the proof of our main result in this paper.
\begin{theorem}\label{T-3-4}
Let $q>p\geq 2$ and let $\xi\in L^q(\Omega)$ and $f\in H_{\mathcal{F}}^q([0,T])$. Assume that the processes $\alpha$, $\beta$ and $\gamma$ satisfy the condition \textbf{(H1)} and that the random variables $\xi\rho_T$ and $\int_0^T\rho_tf_tdt$ belong to $M^q$, where the process $\rho$ is defined in \eqref{joint1}. Then the linear BDSDE \eqref{lbdsde} has a unique solution $(Y,Z)$, and there exists a constant $K>0$ such that
\begin{equation}\label{e-3-7}
\mathbb{E}|Y_t-Y_s|^p\leq K|t-s|^{\frac{p}{2}},
\end{equation}
for all $s, t\in[0,T]$.
\end{theorem}

\subsection{The Malliavin calculus for BDSDEs and the path regularity of $Z$}
In this subsection, we consider the Malliavin calculus for the BDSDE \eqref{bdsde}. First, we make the following assumptions on the terminal value $\xi$
and generator $f$.

\medskip
\noindent
{\bf Assumption (A):}
 Fix $2\leq p<\frac{q}{2}$.
\begin{itemize}
\item[(i)]      $\xi \in \mathbb{D}^{2,q}$, and there exists $L>0$,
such that for all $\theta,\,\theta ^{\prime }\in \lbrack 0,T]$,
\begin{equation}
\mathbb{E}|D_{\theta }\xi -D_{\theta ^{\prime }}\xi |^{p}\leq
L|\theta -\theta ^{\prime }|^{\frac{p}{2}},\,  \label{e5}
\end{equation}%
\begin{equation}
\sup_{0\leq \theta\leq T}\mathbb{E}|D_{\theta}\xi |^{q}<\infty ,
\label{e2}
\end{equation}%
and
\begin{equation}\label{e2-2}
\sup_{0\leq \theta \leq T}\sup_{0\leq u\leq
T}\mathbb{E}|D_{u}D_{\theta }\xi |^{q}<\infty .
\end{equation}

\item[(ii)]   The generator   $f(t,y,z)$ has continuous and uniformly bounded first
and second order partial derivatives with respect to $y$ and $z$,
and  $f(\cdot,0,0)\in H_{ \mathcal{F}}^{q}([0,T])$.

\item[(iii)] The function $g$ has continuous and bounded first and second order derivatives $g'$ and $g''$ respectively.

\item[(iv)]  Assume that $\xi$ and $f$ satisfy  the above conditions (i) and (ii). Let $(Y,Z)$ be the unique solution to
(\ref{bdsde}) with terminal value  $\xi$ and generator $f$.  For each $(y,z)\in\mathbb{R}\times\mathbb{R}$, $f(\cdot ,y,z)$, $\partial _{y}f(\cdot ,y,z)$%
, and $\partial _{z}f(\cdot ,y,z)$ belong to
$\mathbb{L}_{a}^{1,q}$, and the corresponding  Malliavin derivatives $Df(\cdot
,y,z)$, $D\partial _{y}f(\cdot ,y,z)$, and $D\partial _{z}f(\cdot
,y,z)$ satisfy
\begin{eqnarray}
&&\sup_{0\leq \theta \leq T}\EE\left( \int_{\theta}^{T} |D_{\theta
}f(t,Y_t,Z_t)|^{2}dt\right) ^{\frac{q}{2}}<\infty, \label{e3}\\ 
&&\sup_{0\leq \theta \leq T}\mathbb{E}\left(
\int_{\theta}^{T}|D_{\theta }\partial
_{y}f(t,Y_t,Z_t)|^{2}dt\right) ^{\frac{q}{2}}<\infty
\,,\label{e5-1}\\
&& \sup_{0\leq \theta \leq T}\mathbb{E}\left(
\int_{\theta}^{T}|D_{\theta }\partial
_{z}f(t,Y_t,Z_t)|^{2}dt\right) ^{\frac{q}{2}}<\infty \,,\label{e6}
\end{eqnarray}
and there exists $L>0$ such that for any $t\in (0,T]$, and for any
$0\le\theta ,\, \theta ^{\prime }\leq t\le T$
\begin{equation}
\mathbb{E}\left( \int_{t}^{T}|D_{\theta }f(r,Y_r,Z_r)-D_{\theta
^{\prime }}f(r,Y_r,Z_r)|^{2}dr\right) ^{\frac{p}{2}}\leq L|\theta
-\theta ^{\prime }|^{\frac{p}{2}}. \label{e4}
\end{equation}
For each $\theta \in \lbrack 0,T]$, and each pair of $(y,z)$,
$D_{\theta }f(\cdot ,y,z)\in \mathbb{L}_{a}^{1,q}$  and it has
continuous partial derivatives
with respect to $y,z$, which are denoted by $\partial_yD_\theta f(t,y,z)$%
and $\partial_zD_\theta f(t,y,z)$, and the Malliavin derivative
$D_{u}D_{\theta }f(t,y,z)$ satisfies
\begin{equation}\label{e6-2}
\sup_{0\leq \theta \leq T}\sup_{0\leq u\leq T} \mathbb{E}\left(
\int_{\theta\vee u}^{T}|D_{u}D_{\theta }f(t,Y_t,Z_t)|^{2}dt\right)
^{\frac{q}{2}}<\infty .
\end{equation}
\end{itemize}

\bigskip

The following property is easy to check and we omit the proof.

\begin{remark} Conditions (\ref{e5-1}) and (\ref{e6}) imply
\begin{empheq}[left=\empheqlbrace]{align}
&\sup_{0\leq \theta \leq T}\mathbb{E}\left(
\int_{\theta}^{T}|\partial _{y}D_{\theta
}f(t,Y_t,Z_t)|^{2}dt\right) ^{\frac{q}{2}}<\infty \,,\\ 
&
\sup_{0\leq \theta \leq T}\mathbb{E}\left(
\int_{\theta}^{T}|\partial _{z}D_{\theta
}f(t,Y_t,Z_t)|^{2}dt\right) ^{\frac{q}{2}}<\infty \,.
\end{empheq} 
\end{remark}

We refer to \cite[Section 2.4]{HNS10} for several examples where Assumption {\bf (A)} is satisfied, including the cases where $\xi$ is a multiple stochastic integral, a twice Fr\'echet differentiable function of $W$ and a nonnecessarily Lipschitz function of  the  trajectories of a forward diffusion.

The  following  is the main result in this subsection.
\begin{theorem}\label{t.3.1}
Let Assumption  {\bf (A)}  be satisfied. \begin{itemize}
\item[(a)]  Suppose that $(Y,Z)$ is the unique solution pair  in 
$S^q_{\mathcal{F}}(0,T]) \times H^q_{\mathcal{F}}(0,T])$
to the BDSDE (\ref{bdsde}). Then, $Y$ and $Z$ are in
 $ \mathbb{L}^{1,q}_a$ and there exists a version of the Malliavin derivatives
 $\{(D_\theta Y_t,\,D_\theta Z_t)\}_{\,0\leq\theta,\,t\leq T}$ of the solution pair that
 satisfies the following linear BDSDE:
\begin{eqnarray}
D_\theta Y_t &=&D_\theta \xi+\int_t^T[\partial_y{f(r,Y_r,Z_r)}
D_\theta
Y_r+\partial_z{f(r,Y_r,Z_r)}D_\theta Z_r +D_\theta f(r,Y_r,Z_r)]dr\nonumber\\
&&  +\int_t^Tg'(Y_r)D_\theta Y_rd\overleftarrow{B}_r-\int_t^TD_\theta
Z_rdW_r,\quad 0\le\theta\leq t\leq T\,; \label{e.3.12} \\
D_\theta Y_t&=&0,\ D_\theta Z_t \ =\ 0,\quad 0\le t<\theta\leq
T.\label{e.3.12-2}
\end{eqnarray}
Moreover, $\{D_{t} Y_t\}_{0\leq t\leq T}$ defined by
(\ref{e.3.12}) gives a version of $\{Z_t\}_{0\leq t\leq T}$,
namely, $\mu \times \mathbb{P} $ a.e.
\begin{equation}
Z_t=D_{t} Y_t\,.\label{e.3.13}
\end{equation}
\item[(b)] There exists a constant $K>0$, such that, for all $%
s,\, t\in [0,T]$,
\begin{equation}\label{e-z}
\mathbb{E}|Z_{t}-Z_{s}|^{p}\leq K|t-s|^{\frac{p}{2}}.
\end{equation}
\end{itemize}
\end{theorem}

\begin{remark}\label{r.3.8}
From Theorem \ref{t.3.1}  we know that $\{(D_\theta Y_t, D_\theta
Z_t)\}_{0\le \theta\le t\le T}$ satisfies Equations (\ref{e.3.12}) and \eqref{e.3.12-2}
and   $Z_t=D_tY_t$, $\mu\times P$ a.e. Moreover,   since
(\ref{e2}) and (\ref{e3})  hold, we can apply the estimate
(\ref{e.2.1})   in Theorem \ref{T.2.1}  to the linear BDSDE
(\ref{e.3.12})-\eqref{e.3.12-2} and deduce $\sup_{0\le t\le
T}\mathbb{E}|Z_t|^{q}<\infty$. Therefore,  by Corollary \ref{l.3.7},
the process $Y$ satisfies the  inequality (\ref{Hy}). By
Kolmogorov's continuity criterion, this implies that  $Y$ has
H\"{o}lder continuous trajectories of order $\gamma$ for any
$\gamma<\frac 12-\frac 1q$.
\end{remark}

\subsection{An implicit numerical scheme for \eqref{bdsde}}

In this subsection, we consider an implicit numerical scheme for the BDSDE \eqref{bdsde}. By using the path regularity of the process $Z$ in \eqref{e-z}, we are able to give an estimate on the error in $L^p$-sense.

We will need the following fixed point result  in the construction of  our Euler scheme for the BDSDE \eqref{bdsde}. Its proof is easy to obtain, so we omit it here.

\begin{remark}\label{r-3-9}
Let $f$ satisfy (ii) in Assumption  {\bf (A)}, and let $h>0$ be a constant with $h L<1$, where $L$ is the Lipschitz constant for $f$. For any given $\eta,\, z\in\mathbb{R}$ and $t\in[0,T]$, there exists a unique $y=y(\omega)$, such that, 
\[
y=\eta+hf(t,y,z),\ \mbox{a.s.}.
\]
\end{remark}

\bigskip

Let $\pi=\{  0=t_0<t_1<\dots< t_n=T\}$ be a partition of the interval $[0,T]$ and
$|\pi|=\max\limits_{0\leq i\leq n-1}|t_{i+1}-t_i|$. Denote $\Delta_i=t_{i+1}-t_i$, $\Delta B_i=B_{t_{i+1}}-B_{t_i}$ and $\Delta W_i=W_{t_{i+1}}-W_{t_i}$, %
 $0\leq i\leq n-1$. We also assume that 
 \begin{equation}\label{s-3-20}
 |\pi |L<1,
 \end{equation}
 where $L$ is the Lipschitz constant  of the generator $f$.

From the BDSDE (\ref{bdsde}), we know that, when $t\in
[t_i,\,t_{i+1}]$,
\begin{equation}  \label{e.4.1}
Y_t=Y_{t_{i+1}}+\int_t^{t_{i+1}}f(r,Y_r,Z_r)dr+\int_t^{t_{i+1}}g(Y_r)d\overleftarrow{B}_r-\int_t^{t_{i+1}}Z_rdW_r.
\end{equation}

We consider a  numerical  scheme  similar to that introduced in  \cite{BBMM}. For this scheme, we are able to achieve an estimate on the error in $p$-th moment, which is better than the estimates existing in the literature.

The numerical scheme we consider is as follows:
\begin{equation}\label{s-3-22}
Y^{\pi}_{t_n}=\xi^\pi,\  Z_{t_n}^\pi=0,
\end{equation}
and for { $i=n-1, n-2, \dots, 1, 0$}, we define  $ Y^{\pi}_{t_i}$ as follows
\begin{eqnarray}
Z^\pi_{t_i}&=&\frac{1}{\Delta_i}\mathbb{E}\left(Y_{t_{i+1}}^{\pi}\Delta W_i+g(Y^\pi_{t_{i+1}})\Delta B_i\Delta W_i\bigg|\mathcal{F}_{t_i}\right),\label{nu-z}\\
 Y^{\pi}_{t_i}&=&\mathbb{E}\left(Y_{t_{i+1}}^{\pi}+g(Y^\pi_{t_{i+1}})\Delta B_i\bigg|\mathcal{F}_{t_i}\right)
+f\left(t_{i},Y_{t_{i}}^{\pi} , Z^{\pi}_{t_{i}}
\right)\Delta_i,\label{nu-y}
\end{eqnarray}
where $\xi^\pi\in L^p(\Omega)$ is an approximation of the terminal
condition $\xi$.  Then Remark \ref{r-3-9}, \eqref{nu-z} and \eqref{nu-y} lead to a backward recursive formula for the sequence
$\{Y^{\pi} _{t_i},Z^{\pi} _{t_i}\}_{0\le i\le n}$.

Next, for each partition $\pi$, we introduce $(Y^{1,\pi},Z^{1,\pi})$ and give a connection between the approximation solution $(Y^\pi,Z^\pi)$ and the new defined approximation $(Y^{1,\pi},Z^{1,\pi})$. 
More precisely, we proceed as follows. Once $Y_{t_{i+1}}^{\pi}$ and $%
Z^{\pi}_{t_{i+1}}$, which are $\mathcal{F}_{t_{i+1}}$-measurable, are defined, then, for $t\in [t_i,t_{i+1}]$,   we set
\[
Y_t^{1,\pi}=\mathbb{E}\left(Y_{t_{i+1}}^\pi+g(Y_{t_{i+1}}^\pi)\Delta B_i\bigg|\mathcal{F}_t\vee\mathcal{F}_{t_i,T}^B\right).
\] 
By the stochastic integral representation, we have
\begin{equation}\label{e.3.23-s}
Y_{t}^{1,\pi} =Y_{t_{i+1}}^\pi+g(Y_{t_{i+1}}^\pi)\Delta B_i-\int_{t}^{t_{i+1}}Z^{1,\pi}_rdW_r,\quad t\in[t_i,t_{i+1}],
\end{equation}
where $Y_t^{1,\pi}$ and $Z_t^{1,\pi}$ are $\mathcal{F}_t\vee \mathcal{F}_{t_i,T}^B$-measurable for all $t\in[t_i, t_{i+1}]$. 
In particular, at the endpoint $t=t_i$, it holds that
\begin{equation}\label{e.3.24}
Y_{t_i}^{1,\pi}=\mathbb{E}\left(Y_{t_{i+1}}^\pi+g(Y_{t_{i+1}}^\pi)\Delta B_i\bigg|\mathcal{F}_{t_i}\right).
\end{equation}
Hence  from \eqref{nu-z} and \eqref{e.3.23-s} (with $t=t_i$)   it follows 
\begin{eqnarray} 
Z_{t_i}^\pi
&=&\frac{1}{\Delta_i}\mathbb{E}\left(\left[Y_{t_{i+1}}^{\pi} +g(Y^\pi_{t_{i+1}})\Delta B_i\right] \Delta W_i\bigg|\mathcal{F}_{t_i}\right)\nonumber\\
&=&\frac{1}{\Delta_i}\mathbb{E}\left(\left[Y_{t_{i}}^{1,\pi} +\int_{t_i}^{t_{i+1}}Z^{1,\pi}_rdW_r\right]
\Delta W_i \bigg|\mathcal{F}_{t_i}\right)\nonumber\\
&=&\frac{1}{\Delta_i}\mathbb{E}\left(\int_{t_i}^{t_{i+1}}Z_r^{1,\pi}dr\bigg|\mathcal{F}_{t_i}\right). \label{e.3.25}
\end{eqnarray}
Then, by \eqref{nu-y} and \eqref{e.3.24}, the connection between $Y_{t_i}^\pi$ and $Y_{t_i}^{1,\pi}$ is given by 
\begin{equation}\label{e.3.26}
Y_{t_i}^\pi=Y_{t_i}^{1,\pi}+f(t_i, Y_{t_i}^\pi, Z_{t_i}^\pi)\Delta_i.
\end{equation}
 Thus, from \eqref{e.3.23-s} with $t=t_i$ and \eqref{e.3.26}, we have
\begin{equation}\label{s-3-28}
Y_{t_i}^\pi=Y_{t_{i+1}}^\pi+f(t_i, Y_{t_i}^\pi, Z_{t_i}^\pi)\Delta_i+g(Y_{t_{i+1}}^\pi)\Delta B_i-\int_{t_i}^{t_{i+1}}Z^{1,\pi}_rdW_r,\quad i=n-1,\dots, 0.
\end{equation}

\begin{theorem}\label{t-3-10}
Let Assumption  {\bf (A)}  be satisfied, and let the partition $\pi$ satisfy \eqref{s-3-20}. Consider the approximation scheme \eqref{s-3-22}-\eqref{nu-z}. Assume that $\xi^\pi\in L^p(\Omega)$ and that there exists a constant $L_1>0$ such that
\begin{equation}\label{s-3-25}
|f(t_2,y,z)-f(t_1,y,z)|\leq L_1|t_2-t_1|^{\frac{1}{2}},\ a.s.
\end{equation}
for all $t_1,t_2\in[0,T]$ and $y,z\in\mathbb{R}$. Then, there are positive constants $K$ and $\delta$, independent of the partition $\pi$, such that, if $|\pi|<\delta$, then
\begin{equation}\label{s-3-34-1}
\mathbb{E}\max\limits_{0\leq i\leq n-1}|Y_{t_i}-Y_{t_i}^\pi|^p+\mathbb{E}\left(\int_0^T|Z_r-Z_r^{1,\pi}|^2dr\right)^{\frac{p}{2}}\leq K\left(\mathbb{E}|\xi-\xi^\pi|^p+|\pi|^{\frac{p}{2}}\right).
\end{equation}
\end{theorem}

\section{Proofs of Theorem \hbox{\ref{T.2.1}} and Corollary \ref{l.3.7}}
\begin{proof}[Proof of Theorem \hbox{\ref{T.2.1}}] 
The proof of the existence and uniqueness of a solution $(Y,Z)\in S_{\mathcal{F}}^{q}([0,T])\times H_{\mathcal{F}}^{q}([0,T])$
can be found in \cite[Theorem 1.1]{PP94}. 

Let us show the estimate  \eqref{e.2.1}.   {  We first consider the BDSDE \eqref{bdsde} on a fixed interval $[a, b]$ for any $0\leq a<b\leq T$.
From \eqref{bdsde} we get, for any $t\in[a,b]$,
\begin{eqnarray*}
Y_t &=&\mathbb{E}\left(\left.Y_b+\int_t^bf(r,Y_r,Z_r)dr+\int_t^bg(Y_r)d\overleftarrow{B}_r\right| \mathcal{G}_t\right)\\
&=&   \mathbb{E}\left(\left.Y_b+\int_a^bf(r,Y_r,Z_r)dr+\int_a^bg(Y_r)d\overleftarrow{B}_r\right| \mathcal{G}_t\right)\\
&&-\int_a^tf(r,Y_r,Z_r)dr-\int_a^bg(Y_r)d\overleftarrow{B}_r +\int_t^bg(Y_r)d\overleftarrow{B}_r.
\end{eqnarray*}
The above conditional expectation and $\int_t^bg(Y_r)d\overleftarrow{B}_r$ are martingales if they are considered as processes indexed by $t$ for $t\in[a,b]$. Using Doob's maximal inequaliy, the Burkholder-Davis-Gundy inequality, the isometry property for It\^{o}'s integral and the Lipschitz conditions on $f$ and $g$, we have
\begin{eqnarray*} 
\mathbb{E}\sup_{a\leq t\leq b}|Y_t|^q&\leq& C\mathbb{E}|Y_b|^q+C\mathbb{E}\left(\int_a^b|f(r,Y_r,Z_r)|dr\right)^q+C\mathbb{E}\left|\int_a^bg(Y_r)d\overleftarrow{B}_r\right|^q\nonumber\\
&&+C\mathbb{E}\left(\int_a^b|g(Y_r)|^2dr\right)^{\frac{q}{2}}\nonumber\\
&\leq& C\mathbb{E}|Y_b|^q+C(b-a)^{\frac{q}{2}}\mathbb{E}\left(\int_a^b|f(r,0,0)|^2dr\right)^{\frac{q}{2}}+C(b-a)^q\mathbb{E}\sup_{a\leq t\leq b}|Y_t|^q\nonumber\\
&&+C(b-a)^{\frac{q}{2}}\mathbb{E}\left(\int_a^b|Z_r|^2dr\right)^{\frac{q}{2}}+C\mathbb{E}\left(\int_a^b|g(Y_r)|^2dr\right)^{\frac{q}{2}}\nonumber\\
&\leq&C\mathbb{E}|Y_b|^q+C(b-a)^{\frac{q}{2}}\mathbb{E}\left(\int_a^b|f(r,0,0)|^2dr\right)^{\frac{q}{2}}+C(b-a)^q\mathbb{E}\sup_{a\leq t\leq b}|Y_t|^q\nonumber\\
&&+C(b-a)^{\frac{q}{2}}\mathbb{E}\left(\int_a^b|Z_r|^2dr\right)^{\frac{q}{2}}+C(b-a)^{\frac{q}{2}}|g(0)|^q+C(b-a)^{\frac{q}{2}}\mathbb{E}\sup_{a\leq t\leq b}|Y_t|^q,\nonumber\\
\end{eqnarray*}
where $C$, in the above inequalities and in the sequel, is a generic constant independent of $a$ and $b$, which may vary from line to line.
Thus, we have 
\begin{eqnarray}\label{e-o-4-1}
\mathbb{E}\sup_{a\leq t\leq b}|Y_t|^q
&\leq& C\mathbb{E}|Y_b|^q+C(b-a)^{\frac{q}{2}}\mathbb{E}\left(\int_a^b|f(r,0,0)|^2dr\right)^{\frac{q}{2}}+C(b-a)^{\frac{q}{2}}\mathbb{E}\sup_{a\leq t\leq b}|Y_t|^q\nonumber\\
&&+C(b-a)^{\frac{q}{2}}\mathbb{E}\left(\int_a^b|Z_r|^2dr\right)^{\frac{q}{2}}+C(b-a)^{\frac{q}{2}}|g(0)|^q\,. 
\end{eqnarray}
By the Burholder-Davis-Gundy inequality, one has
\begin{equation}\label{e-o-4-2}
\mathbb{E}\left(\int_a^b|Z_r|^2dr\right)^{\frac{q}{2}}\leq c_q\mathbb{E}\left|\int_a^bZ_rdW_r\right|^q,
\end{equation}
for some positive constant $c_q$ depending only on $q$.
From \eqref{bdsde}, one also has
\begin{equation}\label{e-o-4-3}
\int_a^bZ_rdW_r=Y_b-Y_a+\int_a^bf(r,Y_r,Z_r)dr+\int_a^bg(Y_r)d\overleftarrow{B}_r.
\end{equation}
From \eqref{e-o-4-2}, \eqref{e-o-4-3}, the isometry property of It\^{o}'s integral and the Lipschitz conditions on $f$ and $g$, one can write
\begin{eqnarray}\label{e-o-4-4}
\mathbb{E}\left(\int_a^b|Z_r|^2dr\right)^{\frac{q}{2}}&\leq& C\mathbb{E}|Y_b|^q+C\mathbb{E}|Y_a|^q+C\mathbb{E}\left(\int_a^b|f(r,Y_r,Z_r)|dr\right)^q+C\mathbb{E}\left|\int_a^bg(Y_r)d\overleftarrow{B}_r\right|^q\nonumber\\
&\leq& C\mathbb{E}|Y_b|^q+C\mathbb{E}\sup_{a\leq t\leq b}|Y_t|^q+C(b-a)^{\frac{q}{2}}\mathbb{E}\left(\int_a^b|f(r,0,0)|^2dr\right)^{\frac{q}{2}}\nonumber\\
&&+C(b-a)^q\mathbb{E}\sup_{a\leq t\leq b}|Y_t|^q+C(b-a)^{\frac{q}{2}}\mathbb{E}\left(\int_a^b|Z_r|^2dr\right)^{\frac{q}{2}}\nonumber\\
&&+C\mathbb{E}\left(\int_a^b|g(Y_r)|^2dr\right)^{\frac{q}{2}}\nonumber\\
&\leq& C_1\mathbb{E}|Y_b|^q+C_1\mathbb{E}\sup_{a\leq t\leq b}|Y_t|^q+C_1(b-a)^{\frac{q}{2}}\mathbb{E}\left(\int_a^b|f(r,0,0)|^2dr\right)^{\frac{q}{2}}\nonumber\\
&&+C_1(b-a)^q\mathbb{E}\sup_{a\leq t\leq b}|Y_t|^q+C_1(b-a)^{\frac{q}{2}}\mathbb{E}\left(\int_a^b|Z_r|^2dr\right)^{\frac{q}{2}}\nonumber\\
&&+C_1(b-a)^{\frac{q}{2}}|g(0)|^q+C_1(b-a)^{\frac{q}{2}}\mathbb{E}\sup_{a\leq t\leq b}|Y_t|^q,
\end{eqnarray}
where $C_1$ is a positive constant independent of $a$ and $b$.
If $C_1(b-a)^{\frac{q}{2}}<\frac12$, then from \eqref{e-o-4-4} we have
\begin{eqnarray}\label{e-o-4-5}
\mathbb{E}\left(\int_a^b|Z_r|^2dr\right)^{\frac{q}{2}}&\leq& 2C_1\mathbb{E}|Y_b|^q+2C_1\mathbb{E}\sup_{a\leq t\leq b}|Y_t|^q+2C_1(b-a)^q\mathbb{E}\sup_{a\leq t\leq b}|Y_t|^q\nonumber\\
&&+2C_1(b-a)^{\frac{q}{2}}\mathbb{E}\sup_{a\leq t\leq b}|Y_t|^q+2C_1(b-a)^{\frac{q}{2}}\mathbb{E}\left(\int_a^b|f(r,0,0)|^2dr\right)^{\frac{q}{2}}\nonumber\\
&&+2C_1(b-a)^{\frac{q}{2}}|g(0)|^q.
\end{eqnarray}
Substituting    \eqref{e-o-4-5} into \eqref{e-o-4-1} yields 
\begin{eqnarray}\label{e-o-4-6}
\mathbb{E}\sup_{a\leq t\leq b}|Y_t|^q
&\leq&C_2(1+(b-a)^{\frac{q}{2}})\mathbb{E}|Y_b|^q+C_2(b-a)^{\frac{q}{2}}\mathbb{E}\sup_{a\leq t\leq b}|Y_t|^q\nonumber\\
&&+C_2(b-a)^{\frac{q}{2}}\mathbb{E}\left(\int_a^b|f(r,0,0)|^2dr\right)^{\frac{q}{2}}+C_2(b-a)^{\frac{q}{2}}|g(0)|^q, 
\end{eqnarray}
for some positive constant $C_2$ independent of $a$ and $b$. 

If $C_2(b-a)^{\frac{q}{2}}<\frac12$, then we have
\begin{eqnarray}\label{e-o-4-7}
\mathbb{E}\sup_{a\leq t\leq b}|Y_t|^q
&\leq&2C_2(1+(b-a)^{\frac{q}{2}})\mathbb{E}|Y_b|^q+2C_2(b-a)^{\frac{q}{2}}\mathbb{E}\left(\int_a^b|f(r,0,0)|^2dr\right)^{\frac{q}{2}}\nonumber\\
&&+2C_2(b-a)^{\frac{q}{2}}|g(0)|^q.
\end{eqnarray}
Denote 
\[
\Theta_{a,b,q} :=\mathbb{E}\sup_{a\leq t\leq b}|Y_t|^q+\mathbb{E}\left(\int_{a}^{b}|Z_r|^2dr\right)^{\frac{q}{2}}\,. 
\]
We choose a positive constant $\delta$ such that
\begin{eqnarray*}
C_1\delta^{\frac{q}{2}}< \frac12,\qquad 
C_2\delta^{\frac{q}{2}}<\frac12.
\end{eqnarray*}
If $b-a\leq \delta$, then  from \eqref{e-o-4-5} and \eqref{e-o-4-7}  it follows that 
 there exists a positive constant $C_3$ independent of $a$ and $b$ such that
\begin{equation}\label{e-o-4-8}
\Theta_{a,b,q}\leq C_3\left(\mathbb{E}|Y_b|^q+\mathbb{E}\left(\int_a^b|f(r,0,0)|^2dr\right)^{\frac{q}{2}}+|g(0)|^q\right).
\end{equation}

Now, let $l=\left[\frac{T}{\delta}\right]+1$ and $t_i=\frac{iT}{l}$ for $i=0, 1,\dots, l$.   
  By
  \eqref{e-o-4-8} we have on the interval $[t_{l-1},t_l]$ 
\begin{equation}\label{e-o-4-9}
 \Theta_{t_{l-1}, t_l, q} 
\leq C_3\left(\mathbb{E}|\xi|^q+\mathbb{E}\left(\int_{t_{l-1}}^{t_l}|f(r,0,0)|^2dr\right)^{\frac{q}{2}}+|g(0)|^q\right).
\end{equation}

On the interval $[t_{l-2}, t_{l-1}]$, 
we have  in a similar way 
\begin{eqnarray*}
 \Theta_{t_{l-2}, t_{l-1}, q} 
 &\leq&  C_3\left(\mathbb{E}|Y_{t_{l-1}}|^q+\mathbb{E}\left(\int_{t_{l-2}}^{t_{l-1}}|f(r,0,0)|^2dr\right)^{\frac{q}{2}}+|g(0)|^q\right)\\
&\leq& C_3^2\left(\mathbb{E}|\xi|^q+\mathbb{E}\left(\int_{t_{l-1}}^{t_l}|f(r,0,0)|^2dr\right)^{\frac{q}{2}}+|g(0)|^q\right)\nonumber\\
&&+C_3\left(\mathbb{E}\left(\int_{t_{l-2}}^{t_{l-1}}|f(r,0,0)|^2dr\right)^{\frac{q}{2}}+|g(0)|^q\right) \,.
\end{eqnarray*}
Or
\begin{eqnarray*}
 \Theta_{t_{l-2}, t_{l-1}, q} 
&\le& C_3^2\mathbb{E}|\xi|^q+C_3^2\mathbb{E}\left(\int_{t_{l-1}}^{t_l}|f(r,0,0)|^2dr\right)^{\frac{q}{2}}\nonumber\\
&&+C_3\mathbb{E}\left(\int_{t_{l-2}}^{t_{l-1}}|f(r,0,0)|^2dr\right)^{\frac{q}{2}} +(C_3^2+C_3)|g(0)|^q.
\end{eqnarray*}
By induction, for $i=1, 2, \dots, l$, one can write
\begin{eqnarray*}
\Theta_{t_{l-i}, t_{l-i+1}, q}
&\leq&C_3^i\mathbb{E}|\xi|^q+\sum_{j=1}^iC_3^{i+1-j}\mathbb{E}\left(\int_{t_{l-j}}^{t_{l-j+1}}|f(r,0,0)|^2dr\right)^{\frac{q}{2}}+\sum_{j=1}^iC_3^j|g(0)|^q\,. 
\end{eqnarray*}
As a consequence, we have 
\begin{eqnarray*}
\Theta_{0, T, q}
&\leq&\sum_{i=1}^l \mathbb{E}\sup_{t_{l-i}\leq t\leq t_{l-i+1}}|Y_t|^q+l^{\frac{q-2}{2}}\sum_{i=1}^l\mathbb{E}\left(\int_{t_{l-i}}^{t_{l-i+1}}|Z_r|^2dr\right)^{\frac{q}{2}}\\
&\leq&C\left(\sum_{i=1}^lC_3^i\ \mathbb{E}|\xi|^q+\sum_{i=1}^l\sum_{j=1}^i C_3^j|g(0)|^q\right)\\
&&+C\sum_{i=1}^l\sum_{j=1}^iC_3^{i+1-j}\mathbb{E}\left(\int_{t_{l-j}}^{t_{l-j+1}}|f(r,0,0)|^2dr\right)^{\frac{q}{2}}\\
&=&C\left(\sum_{i=1}^lC_3^i\ \mathbb{E}|\xi|^q+\sum_{j=1}^l\sum_{i=j}^l C_3^j|g(0)|^q\right)\\
&&+C\sum_{j=1}^l\sum_{i=j}^lC_3^{i+1-j}\mathbb{E}\left(\int_{t_{l-j}}^{t_{l-j+1}}|f(r,0,0)|^2dr\right)^{\frac{q}{2}}\\
&\leq&C\sum_{i=1}^lC_3^i\left(\mathbb{E}|\xi|^q+\sum_{i=1}^l\mathbb{E}\left(\int_{t_{l-j}}^{t_{l-j+1}}|f(r,0,0)|^2dr\right)^{\frac{q}{2}}+l|g(0)|^q\right)\\
&\leq&C\sum_{i=1}^lC_3^i\left(\mathbb{E}|\xi|^q+l\mathbb{E}\left(\int_0^T|f(r,0,0)|^2dr\right)^{\frac{q}{2}}+l|g(0)|^q\right)\\
&\leq & K\left(\mathbb{E}|\xi|^q+\mathbb{E}\left(\int_0^T|f(r,0,0)|^2dr\right)^{\frac{q}{2}}+|g(0)|^q\right),
\end{eqnarray*}
which is the estimate \eqref{e.2.1}.}
\end{proof}

\begin{proof}[Proof of Corollary \ref{l.3.7}]
Without loss of generality  we assume $0\le s\le t\le T$. Let $C>0$ denote a generic constant independent of $s$ and $t$ but depending only 
on $L$, $q$, $T$ and the quantity $\mathbb{E}|\xi|^q+\mathbb{E}\left(\int_0^T|f(r,0,0)|^2dr\right)^{\frac{q}{2}}+|g(0)|^q$, which may vary from line to
line. Since
\begin{equation*}
Y_s=Y_t+\int_s^tf(r,Y_r,Z_r)dr+\int_s^tg(Y_s)d\overleftarrow{B}_r-\int_s^tZ_rdW_r,
\end{equation*}
we have, by the Lipschitz condition on $f$ and $g$,
\begin{eqnarray*}
&&  \mathbb{E}|Y_t-Y_s|^q \\
&=&\mathbb{E}\left|\int_s^tf(r,Y_r,Z_r)dr+\int_s^tg(Y_s)d\overleftarrow{B}_r-%
\int_s^tZ_rdW_r\right|^q \\
&\le&3^{q-1}\left(\mathbb{E}\left|\int_s^tf(r,Y_r,Z_r)dr\right|^q+\mathbb{E}\left|\int_s^tg(Y_s)d\overleftarrow{B}_r\right|^q+\mathbb{E}%
\left|\int_s^tZ_rdW_r\right|^q\right) \\
&\le&C\left(|t-s|^{\frac{q}{2}}\mathbb{E}\left(\int_s^t|f(r,Y_r,Z_r)|^2dr%
\right)^{\frac{q}{2}}+\mathbb{E}\left(\int_s^t|g(Y_s)|^2dr\right)^{\frac{q}{2}}+\mathbb{E}\left(\int_s^t|Z_r|^2dr\right)^{\frac{q}{2}}\right) \\
&\le&C\bigg\{|t-s|^{\frac{q}{2}}\bigg[\mathbb{E}\left(\int_s^t|Y_r|^2dr%
\right)^{\frac{q}{2}}+\mathbb{E}\left(\int_s^t|Z_r|^2dr\right)^{\frac{q}{2}} +\mathbb{E}\left(\int_s^t|f(r,0,0)|^2dr\right)^{\frac{q}{2}}\bigg]\\%
&&+|t-s|^{\frac{q}{2}}\left(|g(0)|^q+\mathbb{E}\sup_{0\leq r\leq T}|Y_r|^q\right)+|t-s|^{\frac{q}{2}}\sup_{0\le r\le T}\mathbb{E}|Z_r|^{q}\bigg\} \\
&\le&C|t-s|^{\frac{q}{2}}.
\end{eqnarray*}
The proof is completed.
\end{proof}

\setcounter{equation}{0}
\section{Proof of the results for linear BDSDEs}
We first study the properties of the process $\rho$ defined in \eqref{joint1}. The following lemmas will be needed to prove Theorems \ref{p-3-3} and \ref{T-3-4}.
\begin{lemma}\label{L-5-1}
Let the processes $\alpha$, $\beta$ and $\gamma$ satisfy the condition \textbf{(H1)}, and let $\rho$ be defined  by  \eqref{joint1}. Then the following properties are true:
\begin{itemize}
\item[(a)] For any $r\in \mathbb{R}$ we have $\mathbb{E}
\left( \sup\limits_{0\leq t\leq T}\rho_t^r\right) <\infty$;
\item[(b)] For any $r>0$ and $0\leq s\leq t\leq T$, we have 
\[
\mathbb{E}|\rho_t-\rho_s|^r\leq  C|t-s|^{\frac r2},
\]
   where $C$ is a positive constant which is independent of $s$ and $t$.
\end{itemize}
\end{lemma}
\begin{proof}
The proof of the part (a)  is analogous  to that  of Lemma 2.4 in \cite{HNS10}.  The only difference is  the backward integral with respect to $d\overleftarrow{B}_s$. For this we just 
need to reverse the time from $T$ to $0$. We  omit the details   here.

Part (b):  {  Using Taylor's expansion for the function $h(x)=e^x$ up to the first order, we have
 \begin{eqnarray}\label{e-5-11-1}
\rho_t-\rho_s = \rho_{s}e^\eta  \left(  \int_{s}^{t} \gamma_r d\overleftarrow{B}_r +  \int_{s}^{t}  \beta_r dW_r +  \int_{s}^{t} \left (\alpha_r -\frac 12 \beta_r^2 -\frac 12 \gamma_r^2\right)dr\right)= \rho_{s}\delta_{s,t}e^\eta ,
 \end{eqnarray}
for some   random variable $\eta$  between $0$ and $\delta_{s,t}$,
where 
\[
\delta_{s,t}=\int_{s}^{t} \gamma_r d\overleftarrow{B}_r +  \int_{s}^{t}  \beta_r dW_r +  \int_{s}^{t} \left (\alpha_r -\frac 12 \beta_r^2 -\frac 12 \gamma_r^2\right)dr\,.
\]

    Using part (a) and H\"{o}lder's inequality we obtain, for any $p>0$,
\begin{eqnarray}\label{e-5-11-2}
\mathbb{E}e^{p\eta}  &=& \mathbb{E}\left(e^{p\eta}  1_{\{\eta\leq 0\}}\right)+\mathbb{E}\left(e^{p\eta}  1_{\{\eta>0\}}\right)
\le  1+\mathbb{E}\left(e^{p\delta_{s,t}}1_{\{\eta>0\}}\right)\nonumber\\
&\leq& 1+\mathbb{E}\left(\sup_{0\leq t\leq T}\rho_t^p\sup_{0\leq s\leq T}\rho_s^{-p}1_{\{\eta>0\}}\right)\le 
 1+\mathbb{E}\left(\sup_{0\leq t\leq T}\rho_t^p\sup_{0\leq s\leq T}\rho_s^{-p}\right)<\infty.
\end{eqnarray}
Thus, by \eqref{e-5-11-1}, \eqref{e-5-11-2},} part (a), H\"{o}lder's inequality, Burkholder-Davis-Gundy inequality and the boundedness of the processes $\alpha$, $\beta$ and $\gamma$, we can show the result in part (b) as follows
  \begin{eqnarray} \nonumber
  \mathbb{E}|\rho_t-\rho_s|^r&\leq &\left(\mathbb{E}(\rho_s^{2r}e^{2r\eta})\right)^{\frac{1}{2}}\\
  &&\times  \left(\mathbb{E}\left| \int_{s}^{t} \gamma_r d\overleftarrow{B}_r +  \int_{s}^{t}  \beta_r dW_r +  \int_{s}^{t} \left (\alpha_r -\frac 12 \beta_r^2 -\frac 12 \gamma_r^2\right)dr\right|^{2r}\right)^{\frac{1}{2}}\nonumber\\
  &\leq& C\left(\mathbb{E}\left| \int_{s}^{t} \gamma_r d\overleftarrow{B}_r \right|^{2r}+\mathbb{E}\left| \int_{s}^{t} \beta_r dW_r \right|^{2r}+|t-s|^{2r}\right)^{\frac{1}{2}}\nonumber\\
  &\leq& C\left(\mathbb{E}\left| \int_{s}^{t} \gamma_r^2 dr \right|^{r}+\mathbb{E}\left| \int_{s}^{t} \beta_r ^2dr \right|^{r}+|t-s|^{2r}\right)^{\frac{1}{2}}\nonumber\\
  &\leq& C |t-s|^{\frac{r}{2}},\nonumber
  \end{eqnarray}
  where $C$  is a generic constant { independent of $s$ and $t$}.
\end{proof}

Part (b) of Lemma  \ref{L-5-1} implies that for any $0<\varepsilon <\frac 12$ there exists a random variable $G_\varepsilon$ which has moments of all orders, such that, for any $s,t\in [0,T]$,
\begin{equation}  \label{eq5}
|\rho_t -\rho_s | \le G_\varepsilon |t-s|^{\varepsilon}.
\end{equation}

\begin{proof}[Proof of Theorem \hbox{\ref{p-3-3}}] The existence and uniqueness of a solution $(Y,Z)$ to  equation \eqref{lbdsde} follows from Theorem \ref{T.2.1}. Moreover, the solution pair $(Y,Z)$ satisfies the estimate \eqref{e.2.1}.

Now we  are going to show formula \eqref{e-3-5}. Since the process  $\rho$ does not satisfy a stochastic differential equation, this formula cannot be deduced from a version of It\^o's formula for forward and backward stochastic integrals and we need to show this formula by a suitable approximation argument. 
For any $t<T$, we introduce a sequence of  partitions $\pi^n=\{ t=t^n_0<t^n_1 < \dots< t^n_{n} =T\}$  such that 
\[
\lim_{n \to \infty} |\pi^n| =0\,,  
\]
and there exists a constant $K$ such that for all $n$,
\begin{equation}\label{partition-req}
n|\pi^n|\leq K,
\end{equation}
where $|\pi^n|=\max\limits_{0\leq i \leq n-1}\{t^n_{i+1}-t^n_i\}$.  To simplify the notation we omit the superindex  in the partition points and we simply write $t^n_i =t_i$.
 
 Consider the decomposition
 \begin{eqnarray}\label{e-5-11}
 Y_T \rho_T -Y_t \rho_t &=&\sum_{i=0}^{n-1}  \left(Y_{t_{i+1}} \rho_{t_{i+1}}-Y_{t_i} \rho_{t_i} \right)\nonumber\\
 &=& \sum_{i=0}^{n-1}  \left( (Y_{t_{i+1}}-Y_{t_i} )\rho_{t_i}+Y_{t_{i+1}} (\rho_{t_{i+1}} -\rho_{t_i} ) \right)=: A_n+B_n.
 \end{eqnarray}
 For  the term $A_n$ we can write
  \begin{eqnarray}
 A_n&=&\sum_{i=0}^{n-1} \rho_{t_i} 
 \left( \int_{t_i}^{t_{i+1}} (-f_r-\alpha_r Y_r -\beta_rZ_r)dr - \int_{t_i}^{t_{i+1}} \gamma_r Y_r d\overleftarrow{B}_r 
 +\int_{t_i}^{t_{i+1}}  Z_r dW_r \right)\nonumber\\
 &=&A_{n,0}  +  A_{n,1}, \label{e.def-An} 
  \end{eqnarray}
  where
 \begin{equation}
 A_{n,0}= -\sum_{i=0}^{n-1} \rho_{t_i}   \int_{t_i}^{t_{i+1}} \gamma_r Y_r d\overleftarrow{B}_r\, \label{e.def-A-0},
 \end{equation}
 and
 \begin{equation}
  A_{n,1}=\sum_{i=0}^{n-1} \rho_{t_i}\left(\int_{t_i}^{t_{i+1}} (-f_r-\alpha_r Y_r -\beta_rZ_r)dr
 +\int_{t_i}^{t_{i+1}}  Z_r dW_r \right).
\end{equation}
  As we shall see that $A_{n,0}$ will be canceled by a term to be introduced later. As for $A_{n,1}$ it holds that
  \begin{equation}
  \lim_{n\rightarrow \infty} A_{n,1}= \int_t^T(-f_r - \alpha_r Y_r -\beta_rZ_r)\rho_rdr +\int_t^TZ_r  \rho_rdW_r, \ \mbox{in}\ L^1(\Omega). \label{e.an1-limit} 
  \end{equation}
  Indeed, the convergence of the Lebesgue integral follows from the continuity of $\rho $ (see \eqref{eq5}) and the integrability properties of the processes  $f$, $\alpha Y$ and  $\beta Z$. For the stochastic integral, taking into account that  $\rho$ is $\mathcal{G}$-adapted, the process $Z$ is in $H_{\mathcal{G}}^2([0,T])$ due to \eqref{e.2.1}.  Using the Burkholder-Davis-Gundy inequality and inequality \eqref{eq5}, we obtain
 {  \begin{eqnarray*}
  \mathbb{E} \left|  \sum_{i=0}^{n-1} \int_{t_i}^{t_{i+1}}  Z_r( \rho_{t_i} -\rho_r) dW_r \right|& \leq&
   \mathbb{E}   \left| \sum_{i=0}^{n-1} \int_{t_i}^{t_{i+1}}  Z^2_r( \rho_{t_i} -\rho_r)^2 dr \right|^{\frac 12} \\
  &  \le&  \left( \mathbb{E}  \int_0^T Z_r^2 dr \right)^{\frac12}\left( \mathbb{E} G_{\varepsilon}\right)^{\frac12} |\pi_n|^{\varepsilon},
    \end{eqnarray*}}
    for any $\varepsilon \in (0,\frac 12)$.  This proves \eqref{e.an1-limit}.  
  
  Next, let us consider the term $B_n$. Using Taylor's expansion up to the second order, $\rho_{t_{i+1}} - \rho_{t_i} $ can be decomposed as follows
\[
 \rho_{t_{i+1}} - \rho_{t_i} = \rho_{t_i}  \left(  \Psi_i  +\frac 12  \Psi_i^2 + R_i \right),
\]
  where
  \[
  \Psi_i= \int_{t_i}^{t_{i+1}} \gamma_r d\overleftarrow{B}_r +  \int_{t_i}^{t_{i+1}}  \beta_r dW_r +  \int_{t_i}^{t_{i+1}} \left (\alpha_r -\frac 12 \beta_r^2 -\frac 12 \gamma_r^2\right)dr 
  \]
  and the residual term $R_i$ has the form  $
  R_i=\frac 16 \Psi_i^3 e^\eta$, 
  with $\eta$ between $0$ and $ \Psi_i$. It is easy to show that
  \begin{eqnarray*}
 && \sum_{i=0}^{n-1} Y_{t_{i+1}} \rho_{t_i} \Bigg(
  \frac 12  \left( \int_{t_i}^{t_{i+1}} \left (\alpha_r -\frac 12 \beta_r^2 -\frac 12 \gamma_r^2\right)dr \right)^2 \\
  && \quad+
  \left( \int_{t_i}^{t_{i+1}} \left (\alpha_r -\frac 12 \beta_r^2 -\frac 12 \gamma_r^2 \right) dr\right)
  \left(\int_{t_i}^{t_{i+1}} \gamma_r d\overleftarrow{B}_r +  \int_{t_i}^{t_{i+1}}  \beta_r dW_r \right) 
  +R_i \Bigg)
  \end{eqnarray*}
  converges in probability to zero as $n$ tends to infinity. Therefore,
  \[
  \lim_{n\rightarrow \infty } B_n=   \lim_{n\rightarrow \infty } B'_n
  \]
  in probability, where
 \begin{eqnarray}
 B'_n&=&\sum_{i=0}^{n-1}   Y_{t_{i+1}}  \rho_{t_i} \Psi_i
 +\frac 12 \sum_{i=0}^{n-1} Y_{t_{i+1}}\rho_{t_i}  \left(  \int_{t_i}^{t_{i+1}} \gamma_r d\overleftarrow{B}_r +  \int_{t_i}^{t_{i+1}}  \beta_r dW_r  \right)^2\nonumber\\
 &=&\sum_{i=0}^{n-1}Y_{t_{i+1}}\rho_{t_i} \int_{t_i}^{t_{i+1}} \gamma_r d\overleftarrow{B}_r +\sum_{i=0}^{n-1}Y_{t_{i}}\rho_{t_i} \int_{t_i}^{t_{i+1}} \beta_r dW_r+\sum_{i=0}^{n-1}(Y_{t_{i+1}}-Y_{t_i})\rho_{t_i} \int_{t_i}^{t_{i+1}} \beta_r dW_r\nonumber\\
 &&+\sum_{i=0}^{n-1}Y_{t_{i+1}}\rho_{t_i}\int_{t_i}^{t_{i+1}}\left(\alpha_r-\frac{1}{2}\beta_r^2\right)dr+\frac{1}{2}\sum_{i=0}^{n-1}Y_{t_{i+1}}\rho_{t_i}\left[\left(\int_{t_i}^{t_{i+1}}\gamma_rd\overleftarrow{B}_r\right)^2- \int_{t_i}^{t_{i+1}} \gamma_r^2dr\right]\nonumber\\
 &&+\frac{1}{2}\sum_{i=0}^{n-1}Y_{t_{i}}\rho_{t_i}\left(\int_{t_i}^{t_{i+1}}\beta_rdW_r\right)^2+\frac{1}{2}\sum_{i=0}^{n-1}\left(Y_{t_{i+1}}-Y_{t_i}\right)\rho_{t_i}\left(\int_{t_i}^{t_{i+1}}\beta_rdW_r\right)^2\nonumber\\
 &&+\sum_{i=0}^{n-1}Y_{t_{i+1}}\rho_{t_i}\int_{t_i}^{t_{i+1}}\gamma_rd\overleftarrow{B}_r\int_{t_i}^{t_{i+1}}\beta_rdW_r\nonumber\\
 &=:&B_{n,1}+B_{n,2}+B_{n,3}+B_{n,4}+B_{n,5}+B_{n,6}+B_{n,7}+B_{n,8}.
 \end{eqnarray}
 We are going to analyze the asymptotic behavior of  each term in the above decomposition. 
First, one can easily show  that the  following limits hold true:
 \begin{eqnarray}
 \lim\limits_{n\to\infty}B_{n,2}&=&\int_t^T\beta_rY_r\rho_rdW_r, \quad  \mbox{in}\ L^1(\Omega),\\
 \lim\limits_{n\to\infty}B_{n,4}&=&\int_t^T\left(\alpha_rY_r\rho_r-\frac{1}{2}\beta_r^2Y_r\rho_r\right)dr,\quad \mbox{a.s.},\\
 \lim\limits_{n\to\infty}B_{n,6}&=&\frac{1}{2}\int_t^T\beta_r^2Y_r\rho_rdr,\quad  \mbox{in}\ L^1(\Omega).
\end{eqnarray}
Using the fact that 
 \[
 \mathbb{E}\left(\max_{0\leq i\leq n-1}|Y_{t_{i+1}}-Y_{t_i}|^2\right)\to 0, \ \mbox{as}\ n\to \infty,
 \]
 which can be proved by the estimate \eqref{e.2.1} and the dominated convergence theorem, and also using part (a) in Lemma \ref{L-5-1},
we can show that 
\begin{eqnarray}
\mathbb{E}|B_{n,7}|&\leq&\frac{1}{2}\mathbb{E}\left(\max_{0\leq i\leq n-1}|Y_{t_{i+1}}-Y_{t_i}|\sup_{0\leq t\leq T}\rho_t\sum_{i=0}^{n-1}\left(\int_{t_i}^{t_{i+1}}\beta_rdW_r\right)^2\right)\nonumber\\  \nonumber
&\leq&\frac{1}{2}\left( \mathbb{E}\left(\max_{0\leq i\leq n-1}|Y_{t_{i+1}}-Y_{t_i}|^2\right)\right)^{\frac{1}{2}}\left( \mathbb{E}\sup_{0\leq t\leq T}\rho_t^{4}\right)^{\frac 14 } \\
&&
\times \left(\mathbb{E}\left[\sum_{i=0}^{n-1}\left(\int_{t_i}^{t_{i+1}}\beta_rdW_r\right)^2\right]^4\right)^{\frac{1}{4}}\to 0,\nonumber\\
\end{eqnarray}
as $n\to\infty$, since $\lim\limits_{n\to \infty}\mathbb{E}\left[\sum_{i=0}^{n-1}\left(\int_{t_i}^{t_{i+1}}\beta_rdW_r\right)^2\right]^4=\mathbb{E}\left(\int_t^T\beta_r^2dr\right)^4 $.

The term $B_{n,1}$ can be further decomposed as follows
\begin{eqnarray}
B_{n,1}&=&\sum_{i=0}^{n-1} \rho_{t_i}   \int_{t_i}^{t_{i+1}} \gamma_r Y_r d\overleftarrow{B}_r+\sum_{i=0}^{n-1}\rho_{t_i}\left[Y_{t_{i+1}}\int_{t_i}^{t_{i+1}} \gamma_r  d\overleftarrow{B}_r-\int_{t_i}^{t_{i+1}} \gamma_r Y_r d\overleftarrow{B}_r\right]\nonumber\\
&=:& -A_{n,0}+B_{n,1,1},
\end{eqnarray}
where $A_{n,0}$  is defined by \eqref{e.def-A-0}, which 
 is canceled with the term in \eqref{e.def-An}.  
Next, we will show the following two limits hold in $L^1(\Omega)$:
\begin{equation}
\lim\limits_{n\to\infty}B_{n,1,1}= \lim\limits_{n\to\infty}\sum_{i=0}^{n-1}\rho_{t_i}\left[Y_{t_{i+1}}\int_{t_i}^{t_{i+1}} \gamma_r  d\overleftarrow{B}_r-\int_{t_i}^{t_{i+1}} \gamma_r Y_r d\overleftarrow{B}_r\right]=0, \label{e-5-22}
\end{equation}
and
\begin{equation}
\lim\limits_{n\to\infty}B_{n,5}= \lim\limits_{n\to\infty}\frac{1}{2}\sum_{i=0}^{n-1}Y_{t_{i+1}}\rho_{t_i}\left[\left(\int_{t_i}^{t_{i+1}}\gamma_rd\overleftarrow{B}_r\right)^2- \int_{t_i}^{t_{i+1}} \gamma_r^2dr\right] =0.  \label{e-5-23}
\end{equation}

\noindent
{\it Proof of the convergence \eqref{e-5-22}}:
For any positive integer $m<n$, consider  the partition  $\pi^m=\{ t=t_0^m<t^m_1<\dots<t_m^m=T\}$. Define for each $i=0, 1, \dots, n$,
\[
\tau_i=\max\{t_j^m: t_j^m\leq t_i\}\ \mbox{and}\ \sigma_i=\min\{t_j^m: t_j^m>t_i\}.
\]
Then, we can rewrite $B_{n,1,1}$ as follows
\begin{eqnarray}\label{e-5-24}
B_{n,1,1}&=&\sum_{i=0}^{n-1}\rho_{\tau_i}\left[Y_{t_{i+1}}\int_{t_i}^{t_{i+1}} \gamma_r  d\overleftarrow{B}_r-\int_{t_i}^{t_{i+1}} \gamma_r Y_r d\overleftarrow{B}_r\right]\nonumber\\
&&+\sum_{i=0}^{n-1}(\rho_{t_i}-\rho_{\tau_i})\left[Y_{t_{i+1}}\int_{t_i}^{t_{i+1}} \gamma_r  d\overleftarrow{B}_r-\int_{t_i}^{t_{i+1}} \gamma_r Y_r d\overleftarrow{B}_r\right]\nonumber\\
&=&\sum_{j=0}^{m-1}\rho_{t_j^m}\sum_{i, t_j^m\leq t_i<t_{j+1}^m}\left[Y_{t_{i+1}}\int_{t_i}^{t_{i+1}} \gamma_r  d\overleftarrow{B}_r-\int_{t_i}^{t_{i+1}} \gamma_r Y_r d\overleftarrow{B}_r\right]\nonumber\\
&&+\sum_{i=0}^{n-1}(\rho_{t_i}-\rho_{\tau_i})\left[Y_{t_{i+1}}\int_{t_i}^{t_{i+1}} \gamma_r  d\overleftarrow{B}_r-\int_{t_i}^{t_{i+1}} \gamma_r Y_r d\overleftarrow{B}_r\right]\nonumber\\
&=:&B_{n,m, 1}+B_{n,m, 2}.
\end{eqnarray}
If $m$ is fixed, then  for each $j=0,1,\dots, m-1$ we have
\[
\lim _{n\rightarrow \infty} \sum_{i: t_j^m\leq t_i<t_{j+1}^m}\left[Y_{t_{i+1}}\int_{t_i}^{t_{i+1}} \gamma_r  d\overleftarrow{B}_r-\int_{t_i}^{t_{i+1}} \gamma_r Y_r d\overleftarrow{B}_r\right] =0,
\]
in $L^2(\Omega)$, which implies
\begin{equation}\label{e-5-25}
\lim _{n\rightarrow \infty}  B_{n, m, 1}=  0, 
\end{equation}
in $L^2(\Omega)$, for each fixed $m$. 
 
For the term $B_{n,m,2}$, using   H\"{o}lder's inequality, the estimate \eqref{eq5} and the isometry of backward It\^o stochastic integrals, we can write
\begin{eqnarray} \label{Jan-17-1}
\mathbb{E}|B_{n,m,2}|&\leq & \mathbb{E}\left(\max_{0\leq i\leq n-1}|\rho_{t_i}-\rho_{\tau_i}|\sum_{i=0}^{n-1}\left|Y_{t_{i+1}}\int_{t_i}^{t_{i+1}} \gamma_r  d\overleftarrow{B}_r-\int_{t_i}^{t_{i+1}} \gamma_r Y_r d\overleftarrow{B}_r\right|\right)\nonumber\\
&\leq &\left(\mathbb{E}\max_{0\leq i\leq n-1}|\rho_{t_i}-\rho_{\tau_i}|^2\right)^{\frac{1}{2}}\left(\mathbb{E}\left(\sum_{i=0}^{n-1}\left|Y_{t_{i+1}}\int_{t_i}^{t_{i+1}} \gamma_r  d\overleftarrow{B}_r-\int_{t_i}^{t_{i+1}} \gamma_r Y_r d\overleftarrow{B}_r\right|\right)^2\right)^{\frac{1}{2}}\nonumber\\
&\leq&|\pi^m|^\varepsilon{  \left( \mathbb{E} G_\varepsilon^2\right)^{\frac12}  }
\left(n\mathbb{E}\sum_{i=0}^{n-1}\left|Y_{t_{i+1}}\int_{t_i}^{t_{i+1}} \gamma_r  d\overleftarrow{B}_r-\int_{t_i}^{t_{i+1}} \gamma_r Y_r d\overleftarrow{B}_r\right|^2\right)^{\frac{1}{2}}\nonumber\\
&=&|\pi^m|^\varepsilon  { \left( \mathbb{E} G_\varepsilon^2\right)^{\frac12} } \left(n\mathbb{E}\sum_{i=0}^{n-1}\int_{t_i}^{t_{i+1}}\left|\gamma_r(Y_{t_{i+1}}-Y_r)\right|^2dr\right)^{\frac{1}{2}}\nonumber\\
&\le& L |\pi^m|^\varepsilon   { \left( \mathbb{E} G_\varepsilon^2\right)^{\frac12}} \left(n \mathbb{E}\sum_{i=0}^{n-1}\int_{t_i}^{t_{i+1}}|Y_{t_{i+1}}-Y_r|^2dr\right)^{\frac{1}{2}}.
\end{eqnarray}
We are going to make use of the  following estimate for the linear equation \eref{lbdsde} over the interval $[s,t]\subseteq [0,T]$
\begin{equation}\label{Jan-17-2}
\mathbb{E}|Y_t-Y_s|^2\leq C\left(|t-s|+\int_s^t|Z_r|^2dr\right),\quad 0\leq s\leq t\leq T,
\end{equation}
 which can be easily proved by the boundedness of the coefficients $\alpha,\,\beta$ and $\gamma$, and the estimate \eqref{e.2.1} for the linear equation \eqref{lbdsde}.
 From \eqref{Jan-17-1}, \eqref{Jan-17-2} and \eqref{partition-req}, we obtain
 \begin{eqnarray}
\mathbb{E}|B_{n,m,2}|&\leq &
 C    |\pi^m|^\varepsilon   { \left( \mathbb{E} G_\varepsilon^2\right)^{\frac12}}
 \left(n\left(|\pi^n |+\mathbb{E}\sum_{i=0}^{n-1}\int_{t_i}^{t_{i+1}}\int_r^{t_{i+1}}|Z_u|^2dudr\right)\right)^{\frac{1}{2}}\nonumber\\
&\leq&  C    |\pi^m|^\varepsilon   { \left( \mathbb{E} G_\varepsilon^2\right)^{\frac12}}   \left(n|\pi^n|\left(1+\mathbb{E}\int_t^T|Z_u|^2du\right)\right)^{\frac{1}{2}}\nonumber\\
&\leq& C       |\pi^m|^\varepsilon   { \left( \mathbb{E} G_\varepsilon^2\right)^{\frac12}} \left(1+\mathbb{E}\int_t^T|Z_u|^2du\right)^{\frac{1}{2}}, \label{e-5-26}
\end{eqnarray}
which converges to $0$ 
as $m\to\infty$, uniformly in $n$.   Thus, the limit \eqref{e-5-22} follows from \eqref{e-5-24}, \eqref{e-5-25} and \eqref{e-5-26}.

\noindent
{\it Proof of the convergence \eqref{e-5-23}}:
Similarly, we write $B_{n,5}$ as
\begin{eqnarray}\label{e-5-27}
B_{n,5}&=&\frac{1}{2}\sum_{i=0}^{n-1}Y_{\sigma_{i+1}}\rho_{\tau_i}\left[\left(\int_{t_i}^{t_{i+1}}\gamma_rd\overleftarrow{B}_r\right)^2- \int_{t_i}^{t_{i+1}} \gamma_r^2dr\right]\nonumber\\
&&+\frac{1}{2}\sum_{i=0}^{n-1}\left(Y_{t_{i+1}}\rho_{t_i}-Y_{\sigma_{i+1}}\rho_{\tau_i}\right)\left[\left(\int_{t_i}^{t_{i+1}}\gamma_rd\overleftarrow{B}_r\right)^2- \int_{t_i}^{t_{i+1}} \gamma_r^2dr\right]\nonumber\\
&=&\frac{1}{2}\sum_{j=0}^{m-1}\rho_{t_j^m} Y_{t_{j+1}^m} \sum_{i, t_j^m\leq t_i<t_{j+1}^m} \left[\left(\int_{t_i}^{t_{i+1}}\gamma_rd\overleftarrow{B}_r\right)^2- \int_{t_i}^{t_{i+1}} \gamma_r^2dr\right]\nonumber\\
&&+\frac{1}{2}\sum_{i=0}^{n-1}\left(Y_{t_{i+1}}(\rho_{t_i}-\rho_{\tau_i})+(Y_{t_{i+1}}-Y_{\sigma_{i+1}})\rho_{\tau_i}\right)\left[\left(\int_{t_i}^{t_{i+1}}\gamma_rd\overleftarrow{B}_r\right)^2- \int_{t_i}^{t_{i+1}} \gamma_r^2dr\right]\nonumber\\
&=:&B_{n,m,5,1}+B_{n,m,5,2}.
\end{eqnarray}
For any fixed $m$, we have, for each $j=0,1,\dots, m-1$,
\[
\lim_{n \to \infty} \sum_{i, t_j^m\leq t_i<t_{j+1}^m}\left[\left(\int_{t_i}^{t_{i+1}}\gamma_rd\overleftarrow{B}_r\right)^2- \int_{t_i}^{t_{i+1}} \gamma_r^2dr\right]=0,
\] 
in $L^2(\Omega)$, which implies
\begin{equation}\label{e-5-28}
\lim_{n \to \infty} B_{n, m, 5,1}=0
\end{equation}
in $L^1(\Omega)$, for each fixed $m$.

For the term $B_{n,m,5,2}$, appying Cauchy-Schwarz inequality, we can write
\begin{eqnarray*}
\mathbb{E}|B_{n,m,5,2}|
&\leq&\frac{1}{2}\mathbb{E}\Bigg(\max_{0\leq i\leq n-1}\{|Y_{t_{i+1}}||\rho_{t_i}-\rho_{\tau_i}|+|Y_{t_{i+1}}-Y_{\sigma_{i+1}}||\rho_{\tau_i}|\}\\
&&\times
\sum_{i=0}^{n-1}\left|\left(\int_{t_i}^{t_{i+1}}\gamma_rd\overleftarrow{B}_r\right)^2- \int_{t_i}^{t_{i+1}} \gamma_r^2dr\right|\Bigg)\\
&\leq&\frac{1}{2}\left(\mathbb{E}\left(\max_{0\leq i\leq n-1}\{|Y_{t_{i+1}}||\rho_{t_i}-\rho_{\tau_i}|+|Y_{t_{i+1}}-Y_{\sigma_{i+1}}||\rho_{\tau_i}|\}\right)^2\right)^{\frac{1}{2}}\\
&&\times\left(\mathbb{E}\left|\sum_{i=0}^{n-1}\left[\left(\int_{t_i}^{t_{i+1}}\gamma_rd\overleftarrow{B}_r\right)^2- \int_{t_i}^{t_{i+1}} \gamma_r^2dr\right]\right|^2\right)^{\frac{1}{2}}\\
&=:&{  \frac 12 \left( B_{n,m,5,3} \cdot  B_{n,m,5,4} \right)}
\end{eqnarray*}
Clearly, the term $ B_{n,m,5,4}$  is uniformly bounded by a constant  depending only on the constant in  Burkholder's inequality and the bound  of the process $\gamma$. On the other hand,  $ B_{n,m,5,3}$ converges to zero as $m$ tends to infinity, uniformly in $n$. 
 Hence,  we   proved that 
 { \begin{equation}\label{e-5-29}
\lim_{m\to \infty}\mathbb{E}|B_{n,m,5,2}|=0,
 \end{equation}
 uniformly in $n$.}
  Therefore, the limit \eqref{e-5-23} can be proved by \eqref{e-5-27}, \eqref{e-5-28} and \eqref{e-5-29}.

Finally, we consider the remaining terms $B_{n,3}$ and $B_{n,8}$ together. Using the equation satisfied by $Y$  we can write  
 \begin{eqnarray}\label{e-5-30}
 B_{n,3}+B_{n,8}&=& \sum_{i=0}^{n-1}\rho_{t_i} \left( \int_{t_i}^{t_{i+1}} (-f_r-\alpha_r Y_r -\beta_rZ_r)dr 
 +\int_{t_i}^{t_{i+1}}  Z_r dW_r \right)\int_{t_i}^{t_{i+1}}\beta_rdW_r\nonumber\\
 &&+ \sum_{i=0}^{n-1} \rho_{t_i}\left[Y_{t_{i+1}}\int_{t_i}^{t_{i+1}} \gamma_r d\overleftarrow{B}_r-\int_{t_i}^{t_{i+1}} \gamma_r Y_r d\overleftarrow{B}_r \right]\int_{t_i}^{t_{i+1}}\beta_rdW_r\nonumber\\
 &=:&B_{n,3,1}+B_{n,3,2}. 
 \end{eqnarray}
 From the adaptedness of the process $\rho$ to the filtration $\mathcal{G}$ and the classical It\^{o} calculus,  together with the estimates proved in 
 Lemma   \ref{L-5-1}, we can show that
 \begin{equation}\label{e-5-31}
  \lim\limits_{n\to\infty}B_{n,3,1}=\int_t^T\beta_r\rho_rZ_rdr,\ \mbox{in}\ L^1(\Omega).
 \end{equation}
For the term $B_{n,3,2}$,  we have the estimate
\[
|B_{n,3,2}| \le  \sup_{0\leq t\leq T}\rho_t   \left(  \sum_{i=0}^{n-1}\left(\int_{t_i}^{t_{i+1}}\gamma_r(Y_{t_{i+1}}-Y_r)d\overleftarrow{B}_r\right)^2
\right)^{\frac 12}   \left(\sum_{i=0}^{n-1}\left(\int_{t_i}^{t_{i+1}}\beta_rdW_r\right)^2\right)^{\frac{1}{2}}.
\]
The factors $\sup_{0\leq t\leq T}\rho_t$ and  $ \left(\sum_{i=0}^{n-1}\left(\int_{t_i}^{t_{i+1}}\beta_rdW_r\right)^2\right)^{\frac{1}{2}}$ have moments of all orders uniformly bounded in $n$. Therefore, in order to show that
 \begin{equation}\label{e-5-32}
  \lim\limits_{n\to\infty} \mathbb{E} |B_{n,3,2}|=0,
 \end{equation}
 it suffices to prove that
 \[
   \lim\limits_{n\to\infty} \mathbb{E}   \sum_{i=0}^{n-1}\left(\int_{t_i}^{t_{i+1}}\gamma_r(Y_{t_{i+1}}-Y_r)d\overleftarrow{B}_r\right)^2=0,
   \]
   which follows from the following estimate
   \begin{eqnarray*}
   \mathbb{E}   \sum_{i=0}^{n-1}\left(\int_{t_i}^{t_{i+1}}\gamma_r(Y_{t_{i+1}}-Y_r)d\overleftarrow{B}_r\right)^2 &=&
  \mathbb{E}   \sum_{i=0}^{n-1}\int_{t_i}^{t_{i+1}}\gamma_r^2(Y_{t_{i+1}}-Y_r)^2 dr \\
  &\le& L^2 T  \mathbb{E} \left( \sup_{|r-s| \le |\pi^n|} |Y_s-Y_r|^2\right).
   \end{eqnarray*}
From \eqref{e-5-30}, \eqref{e-5-31} and \eqref{e-5-32}, we deduce  that 
\begin{equation}\label{e-5-33}
B_{n,3}+B_{n,8}\to \int_t^T\beta_r\rho_rZ_rdr,\ \mbox{in} \ L^1(\Omega), \ \mbox{as}\ n\to\infty.
\end{equation}
Therefore, from \eqref{e-5-11}-\eqref{e-5-23} and \eqref{e-5-33}, we can prove \eqref{e-3-5}, and hence, \eqref{e-3-6}.
\end{proof}

\begin{proof}[Proof of Theorem \ref{T-3-4}]
The existence and uniqueness of a solution
has been shown in Theorem \hbox{\ref{T.2.1}}, where the estimate  \eqref{e.2.1}  is also proved. On the other hand, we have proved the  following explicit representation for  the process $Y$ (see \eqref{e-3-6}) 
\[
Y_t=\rho_t^{-1}\mathbb{E}\left(\xi\rho_T+\left. \int_t^T\rho_sf_sds\right | \mathcal{G}_t\right)= \mathbb{E}\left( \xi \rho _{t,T}+\left.\int_{t}^{T}\rho _{t,r}f_{r}dr\right|\mathcal{G%
}_{t}\right),
\]
where $\rho _{t,r}=\rho _{t}^{-1}\rho _{r}$ for any $0\le t\leq r\leq T$. For any  $t\in [0,T]$, denote $$\delta _{t}=\rho _{t}^{-1}=\exp\left\{-\int_0^t\beta_sdW_s-\int_0^t\gamma_sd\overleftarrow{B_s}-\int_0^t\left(\alpha_s-\frac{1}{2}\beta^2_s-\frac{1}{2}\gamma_s^2\right)ds\right\}.$$
Then, $\rho_{t,r}=\delta_t\rho_r$ with $0\leq t\leq r\leq T$.

For any $0\leq s\leq t\leq T$ and any positive number $r\geq 1$,
applying Lemma \ref{L-5-1} to the
process $\{\delta _{t}\}_{0\le t\le T}$, we have
\[
\mathbb{E}\sup_{0\leq t\leq T}\delta_t^r\leq C,
\]
and
\[
 \mathbb{E}|\delta _{t}-\delta _{s}|^{r}\leq C(t-s)^{\frac{r}{2}}, 
\]%
where $C$ is a positive constant depending only on $L$, $T$ and $r$.

Once we obtain the representation \eqref{e-3-6} and the above two estimates, the proof of the estimate \eqref{e-3-7}  will be analogous to  the proof of a similar estimate in Theorem 2.3 of \cite{HNS10}.
\end{proof}

\setcounter{equation}{0}

\section{Proof of Theorem \ref{t.3.1}}

\begin{proof}[Proof of Theorem \ref{t.3.1}] Part (a): The existence and uniqueness of the solution $(Y,Z)\in S_{\mathcal{F}}^q([0,T])\times H_{\mathcal{F}}^q([0,T)$ 
are  obtained in Theorem \ref{T.2.1}.  We first show that  $Y,\,Z\in \
\mathbb{L}_{a}^{1,2}$ following a recursive argument similar to that used in
 the proof of Proposition 5.3 in \cite{KPQ97}.
 
  Let $Y^0_t=0$ and $Z_t^0=0$ for all $t\in [0,T]$. It is obvious that $(Y^0,Z^0)$ is in $\mathbb{L}_{a}^{1,2}$. We define a sequence of $\{(Y^n,Z^n)\}_{n=0}^\infty$ as follows:
\[
Y_t^{n+1}=\xi+\int_t^Tf(r,Y_r^n,Z^n_r)dr+\int_t^Tg(Y_r^n)d\overleftarrow{B}_r-\int_t^TZ_r^{n+1}dW_r,\quad 0\leq t\leq T.
\]
In the proof of Theorem 1.1 in \cite{PP94}, the convergence of the sequence $\{(Y^n,Z^n)\}_{n=0}^\infty$ to $(Y,Z)$ in   the norm of  $S_{\mathcal{F}}^2([0,T])\times H_{\mathcal{F}}^2([0,T])$ is established.

If $(Y^n,Z^n)$ is in $\mathbb{L}_{a}^{1,2}$, then $\xi+\int_t^Tf(r,Y_r^n,Z^n_r)dr+\int_t^Tg(Y_r^n)d\overleftarrow{B}_r$ is in $\mathbb{D}^{1,2}$, and hence $$Y_{t}^{n+1}=\mathbb{E}\left(\left.\xi+\int_t^Tf(r,Y_r^n,Z^n_r)dr+\int_t^Tg(Y_r^n)d\overleftarrow{B}_r\right|\mathcal{G}_t\right)$$ is in $\mathbb{D}^{1,2}$, for all $t\in[0,T]$. Note that $D_\theta Y_t^{n+1}=0$ if $0\leq t<\theta\leq T$.
Since
 $$\xi+\int_t^Tf(r,Y_r^n,Z^n_r)dr+\int_t^Tg(Y_r^n)d\overleftarrow{B}_r-Y_t^{n+1}=\int_t^TZ_r^{n+1}dW_r,$$
 it follows from Lemma 5.1 in \cite{KPQ97} that $Z^{n+1}$ is in $L^{1,2}_a$, and furthermore, $D_\theta Z_t^{n+1}=0$ if $0\leq t<\theta\leq T$.
 
 Then, the Malliavin derivatives of $Y^{n+1}$ and $Z^{n+1}$ satisfy the following equation
 \begin{eqnarray*}
D_\theta Y^{n+1}_t &=&D_\theta \xi+\int_t^T[\partial_y{f(r,Y^n_r,Z^n_r)}
D_\theta
Y^n_r+\partial_z{f(r,Y^n_r,Z^n_r)}D_\theta Z^n_r +D_\theta f(r,Y^n_r,Z^n_r)]dr\nonumber\\
&&  +\int_t^Tg'(Y^n_r)D_\theta Y^n_rd\overleftarrow{B}_r-\int_t^TD_\theta
Z^{n+1}_rdW_r,\quad 0\le\theta\leq t\leq T\,; \\
D_\theta Y^{n+1}_t&=&0,\ D_\theta Z^{n+1}_t \ =\ 0,\quad 0\le t<\theta\leq
T.
\end{eqnarray*}
From the convergence of the sequence $\{(Y^n,Z^n)\}_{n=0}^\infty$ to $(Y,Z)$ and the Assumption  {\bf (A)}, by using  techniques  similar to those in the proof of Proposition 5.3 in \cite{KPQ97}, we can prove that the sequence $\{(D_\theta Y^n, D_\theta Z^n)\}_{n=0}^\infty$ converges to $(D_\theta Y, D_\theta Z)$, where $(D_\theta Y, D_\theta Z)$ is the unique solution to \eqref{e.3.12} and \eqref{e.3.12-2}. Note that
 $$\int_0^T(\Vert D_\theta Y\Vert_{S^2}^2+\Vert D_\theta Z\Vert_{H^2}^2)d\theta<\infty$$ follows from the estimate in Theorem \ref{t.3.1} with $q=2$. Hence, $(Y,Z)$ is in $\mathbb{L}_{a}^{1,2}$.

Furthermore, from conditions (\ref{e2}) and (\ref{e3}) and the
estimate in Theorem \ref{T.2.1}, we    obtain
\begin{equation}\label{e.3.19-1}
\sup_{0\le\theta\le T}\left\{\EE\sup_{\theta\le t\le T}|D_\theta
Y_t|^q+\EE\left(\int_\theta^T|D_\theta
Z_t|^2dt\right)^{\frac{q}{2}}\right\}<\infty.
 \end{equation}
Hence, by Proposition 1.5.5 in \cite{N06}, $Y$ and $Z$ belong to
$\mathbb{L}_a^{1,q}$. The representation \eqref{e.3.13} can be proved by using the same technique in the proof of Proposition 5.3 in \cite{KPQ97}.\\

\medskip
\noindent
Part (b): Let $0\le s\leq t\leq T$. In this proof, $C>0$ will be a
constant independent of $s$ and $t$,  and may vary from line to
line.

By the representation (\ref{e.3.13}) we have
\begin{equation}\label{e.3.18}
Z_t-Z_s=D_tY_t-D_sY_s=(D_tY_t-D_sY_t)+(D_sY_t-D_sY_s).
\end{equation}%
From Theorem \ref{T.2.1} and Equation (\ref{e.3.12}) for $\theta=s$
and $\theta'=t$ respectively, we obtain, using conditions
(\ref{e5}) and (\ref{e4}),
\begin{eqnarray}
&&\EE|D_tY_t-D_sY_t|^p+\EE\left(\int_t^T|D_tZ_r-D_sZ_r|^2dr\right)^{\frac{p}{2}}\notag\\
&\le&
C\left[\EE|D_t\xi-D_s\xi|^p+\EE\left(\int_t^T|D_tf(r,Y_r,Z_r)-D_sf(r,Y_r,Z_r)|^2dr\right)^{\frac{p}{2}}\right]\notag\\
&\le& C|t-s|^{\frac{p}{2}}\label{e.3.19}.
\end{eqnarray}
Denote $\alpha_u=\partial_y f(u,Y_u,Z_u)$, $\beta_u=\partial_z
f(u,Y_u,Z_u)$ and $\gamma_u=g'(Y_u)$ for all $u\in[0,T]$. Then, by Assumption  {\bf (A)}
$(ii)$ and $(iii)$, the processes $\alpha$, $\beta$ and $\gamma$ satisfy condition \textbf{(H1)}, and from (\ref{e.3.12}) we
have for $r\in[s,T]$
\begin{eqnarray*} 
&D_sY_r=D_s\xi+\int_r^T[\alpha_uD_sY_u+\beta_uD_sZ_u+D_sf(u,Y_u,Z_u)]du+\int_r^T\gamma_uD_sY_rd\overleftarrow{B}_r-\int_r^TD_sZ_udW_u.
\end{eqnarray*}

 Next, we are going to apply Theorem \ref{T-3-4} to the above linear BDSDE to estimate
$\EE|D_sY_t-D_sY_s|^p$. Fix $p^\prime$ with
$p<p^\prime<\frac{q}{2}$ (notice that $p^\prime<\frac{q}{2}$ is
equivalent to $\frac{p^\prime}{q-p^\prime}<1$). From conditions
(\ref{e2}) and (\ref{e3}), it is obvious that $D_s\xi\in
L^q(\Omega)\subset L^{p^\prime}(\Omega)$ and $D_s f(\cdot,Y,Z)\in
H^q_{\mathcal{F}}([0,T])\subset H^{p^\prime}_{\mathcal{F}}([0,T])$ for any $s\in [0,T]$.

Recall  that the random variable $\rho$ defined in \eqref{joint1}:
\[
\rho_r=\exp\left\{\int_0^r\beta_udW_u+\int_0^r\gamma_rd\overleftarrow{B}_r+\int_0^r\left(\alpha_u-\frac{1}{2}\beta_u^2-\frac{1}{2}\gamma_u^2\right)du\right\},
\]
 is $\mathcal{G}_r$-measurable.

For any $0\le\theta\le r\le T$, let us compute
\begin{eqnarray*}
D_\theta\rho_r&=&\rho_r\bigg\{\int_\theta^r[\partial_{yz}f(u,Y_u,Z_u)D_\theta
Y_u+\partial_{zz}f(u,Y_u,Z_u)D_\theta 
Z_u+D_\theta\partial_{z}f(u,Y_u,Z_u)]dW_u\\
&&+\partial_zf(\theta,Y_\theta,Z_\theta)+\int_\theta^r g''(Y_u)D_\theta Y_ud\overleftarrow{B}_u\\
&&+\int_\theta^r(\partial_{yy}f(u,Y_u,Z_u)-\partial_{yz}f(u,Y_u,Z_u)\beta_u+g'(Y_u)g''(Y_u))D_\theta
Y_udu\\
&&+\int_\theta^r(\partial_{yz}f(u,Y_u,Z_u)-\partial_{zz}f(u,Y_u,Z_u)\beta_u)D_\theta
Z_udu\\
&&+\int_\theta^r(D_\theta\partial_{y}f(u,Y_u,Z_u)-\beta_uD_\theta
\partial_{z}f(u,Y_u,Z_u))du\bigg\}.
\end{eqnarray*}

By the boundedness of the first and second order partial
derivatives of $f$ with respect to $y$ and $z$, the boundedness of $g'$ and $g''$, (\ref{e5-1}),
(\ref{e6}), (\ref{e.3.19-1}), Lemma \ref{L-5-1}, the H\"{o}lder
inequality and the Burkholder-Davis-Gundy inequality, it is easy to
show that for any $p''<q$,
\begin{equation}\label{e.3.21}
\sup_{0\le\theta\le T}\EE\sup_{\theta\le r\le
T}|D_\theta\rho_r|^{p''}<\infty.
\end{equation}
Then, the method to show $\rho_TD_s\xi\in M^{p^\prime}$ and $\int_s^T \rho_u D_sf(u,Y_u,Z_u)du\in M^{p^\prime}$ is exactly the same as that in the proof of Theorem 2.6 in \cite{HNS10}. Together with Theorem \ref{T-3-4}, we are able to show \begin{equation}\label{e.3.23}
\EE\vert D_sY_t-D_sY_s\vert^p\le C|t-s|^{\frac{p}{2}},
\end{equation}
for all $s,\, t\in[0,T]$. Combining  (\ref{e.3.23}) with (\ref{e.3.18}) and
(\ref{e.3.19}), we obtain that there is a constant $K>0$
independent of $s$ and $t$, such that,
\[
\EE\vert Z_t-Z_s\vert^p\le K\vert t-s\vert^{\frac{p}{2}},
\]
for all $s,\, t \in[0,T]$.  This completes the proof of  Theorem \ref{t.3.1}. 
\end{proof}

 \setcounter{equation}{0}

\section{Proof of Theorem \ref{t-3-10}}
We start with a comparison between the values of the process $Y$ and the approximation $Y^\pi$ at the partition points. 
It follows from \eqref{e.4.1} and \eqref{s-3-28} that, for $i=n-1, n-2,\dots, 1, 0$,
\begin{eqnarray*}
Y_{t_i}-Y_{t_i}^\pi&=&Y_{t_{i+1}}-Y_{t_{i+1}}^\pi+\int_{t_i}^{t_{i+1}}[f(r,Y_r,Z_r)-f(t_i,Y_{t_i}^\pi,Z_{t_i}^\pi)]dr   \\
&& +\int_{t_i}^{t_{i+1}}[g(Y_r)-g(Y_{t_{i+1}}^\pi)]d\overleftarrow{B}_r
-\int_{t_i}^{t_{i+1}}[Z_r-Z_r^{1,\pi}]dW_r\\
&=&Y_{t_{i+1}}-Y_{t_{i+1}}^\pi+[f(t_i,Y_{t_i},Z_{t_i})-f(t_i,Y_{t_i}^\pi,Z_{t_i}^\pi)]\Delta_i+[g(Y_{t_{i+1}})-g(Y_{t_{i+1}}^\pi)]\Delta B_i\\
&&+\int_{t_i}^{t_{i+1}}[f(r,Y_r,Z_r)-f(t_i,Y_{t_i},Z_{t_i})]dr+\int_{t_i}^{t_{i+1}}[g(Y_r)-g(Y_{t_{i+1}})]d\overleftarrow{B}_r \\
&& -\int_{t_i}^{t_{i+1}}[Z_r-Z_r^{1,\pi}]dW_r. 
\end{eqnarray*}
This can be written as 
\begin{eqnarray}
\nonumber
Y_{t_i}-Y_{t_i}^\pi &=&\xi-\xi^\pi+\sum_{j=i}^{n-1}[f(t_j,Y_{t_j},Z_{t_j})-f(t_j,Y_{t_j}^\pi,Z_{t_j}^\pi)]\Delta_j+\sum_{j=i}^{n-1}[g(Y_{t_{j+1}})-g(Y_{t_{j+1}}^\pi)]\Delta B_j\nonumber\\
&&+\sum_{j=i}^{n-1}\int_{t_j}^{t_{j+1}}[f(r,Y_r,Z_r)-f(t_j,Y_{t_j},Z_{t_j})]dr \nonumber\\  
&&+\sum_{j=i}^{n-1}\int_{t_j}^{t_{j+1}}[g(Y_r)  
-g(Y_{t_{j+1}})]d\overleftarrow{B}_r-\int_{t_i}^{T}[Z_r-Z_r^{1,\pi}]dW_r\nonumber\\
&=&\xi-\xi^\pi+\sum_{j=i}^{n-1}[f(t_j,Y_{t_j},Z_{t_j})-f(t_j,Y_{t_j}^\pi,Z_{t_j}^\pi)]\Delta_j+\sum_{j=i}^{n-1}[g(Y_{t_{j+1}})-g(Y_{t_{j+1}}^\pi)]\Delta B_j\nonumber\\
&&-\int_{t_i}^{T}[Z_r-Z_r^{1,\pi}]dW_r+R_{t_i}^\pi+G_{t_i}^\pi,  \label{eq7}
\end{eqnarray}
where
\begin{equation}\label{s-7-7}
R_{t_i}^\pi=\sum_{j=i}^{n-1}\int_{t_j}^{t_{j+1}}[f(r,Y_r,Z_r)-f(t_j,Y_{t_j},Z_{t_j})]dr,
\end{equation}
and
\begin{equation}\label{s-7-8}
G_{t_i}^\pi=\sum_{j=i}^{n-1}\int_{t_j}^{t_{j+1}}[g(Y_r)-g(Y_{t_{j+1}})]d\overleftarrow{B}_r.
\end{equation}

The following two lemmas will be needed in the proof of Theorem \ref{t-3-10}.

\begin{lemma}\label{l-7-1}
Let all the conditions in Theorem \ref{t-3-10} be satisfied, and let $R^\pi$ and $G^\pi$  be 
defined in \eqref{s-7-7} and \eqref{s-7-8} respectively. Then, the following estimates hold
\begin{eqnarray}
\mathbb{E}\max\limits_{0\leq i\leq n-1}\left( |R_{t_i}^\pi|^p+ |G_{t_i}^\pi|^p\right)&\leq& K|\pi|^{\frac{p}{2}},\label{s-7-9}\\
\mathbb{E}\max\limits_{0\leq i\leq n-1}\left(\mathbb{E}\left(\left.|R^\pi_{t_i}|\right|\mathcal{G}_{t_i}\right)\right)^p+\mathbb{E}\max\limits_{0\leq i\leq n-1}\left|\mathbb{E}\left(\left.G^\pi_{t_i}\right|\mathcal{G}_{t_i}\right)\right|^p&\leq & K|\pi|^{\frac{p}{2}},\label{s-7-9-1}
\end{eqnarray}
for some constant $K>0$.
\end{lemma}
\begin{proof} In this proof, let $C>0$ be a generic constant depending only on $T$, $p$ and   the constants appearing in the assumptions in Theorem \ref{t-3-10}.

 Define two functions $\{t_1(r)\}_{0\le r\le T}$ and $\{t_2(r)\}_{0\le r\le T}$ by
\begin{equation*}
t_1(r)=
 \begin{cases}
 0 &\text{if $r=0$,}
 \\
 t_{i} & \text{if $t_i< r\leq t_{i+1}$, $i=0,\dots, n-1$,}
\end{cases}
\end{equation*}
and
\begin{equation*}
t_2(r)=
 \begin{cases}
 T &\text{if $r=T$,}
 \\
 t_{j+1} & \text{if $t_j\le r<t_{j+1}$, $j=n-1,\dots,0$.}
\end{cases}
\end{equation*}

 From \eqref{s-7-7}, \eqref{s-3-25},  H\"{o}lder's inequality, the Lipschitz condition on $f$,  Corollary \ref{l.3.7} and Theorem \ref{t.3.1} (b), we obtain
\begin{eqnarray}\label{s-7-10}
\mathbb{E}\left[\max\limits_{0\leq i\leq n-1}|R_{t_i}^\pi|^p\right] &\leq&\mathbb{E}\sup\limits_{0\leq i\leq n-1}\left(\sum_{j=i}^{n-1}\int_{t_j}^{t_{j+1}}\left|f(r,Y_r,Z_r)-f(t_j,Y_{t_j},Z_{t_j})\right|dr\right)^p\nonumber\\
&=&\mathbb{E}\sup\limits_{0\leq i\leq n-1}\left(\int_{t_i}^{T}\left|f(r,Y_r,Z_r)-f(t_1(r),Y_{t_1(r)},Z_{t_1(r)})\right|dr\right)^p\nonumber\\
&\leq&\mathbb{E}\left(\int_{0}^{T}\left|f(r,Y_r,Z_r)-f(t_1(r),Y_{t_1(r)},Z_{t_1(r)})\right|dr\right)^p\nonumber\\
&\leq&T^{p-1}\mathbb{E}\int_0^T|f(r,Y_r,Z_r)-f(t_1(r),Y_{t_1(r)},Z_{t_1(r)})|^pdr\le C|\pi|^{\frac{p}{2}}.
\end{eqnarray}
By  Burkholder-Davis-Gundy's inequality,  condition $(iii)$ in Assumption  {\bf (A)}, Corollary \ref{l.3.7} and Theorem \ref{t.3.1}, we have
\begin{eqnarray}\label{s-7-11}
\mathbb{E}\max\limits_{0\leq i\leq n-1}|G_{t_i}^\pi|^p&=&\mathbb{E}\sup\limits_{0\leq i\leq n-1}\left|\int_{t_i}^T[g(Y_r)-g(Y_{t_2(r)})]d\overleftarrow{B}_r\right|^p
\nonumber \\
&\leq& C\mathbb{E}\left(\int_{0}^T\left|g(Y_r)-g(Y_{t_2(r)})\right|^2dr\right)^{\frac{p}{2}}\nonumber\\
&\leq&CT^{\frac{p-2}{2}}\mathbb{E}\int_0^T\left|g(Y_r)-g(Y_{t_2(r)})\right|^pdr\le C\pi^{\frac{p}{2}}.
\end{eqnarray}
Then, the  estimate \eqref{s-7-9} follows from \eqref{s-7-10} and \eqref{s-7-11}.

Let us now turn to the proof of \eqref{s-7-9-1}.
By Doob's maximal inequality, H\"{o}lder's inequality, \eqref{s-3-25}, the Lipschitz condition on $f$,  Corollary \ref{l.3.7} and Theorem \ref{t.3.1} (b), we get
\begin{eqnarray}\label{s-7-13-1}
&&\mathbb{E}\max\limits_{0\leq i\leq n-1}\left(\mathbb{E}\left(\left.|R^\pi_{t_i}|\right|\mathcal{G}_{t_i}\right)\right)^p\nonumber\\
&\leq&\mathbb{E}\left(\max\limits_{0\leq i\leq n-1}\mathbb{E}\left(\left.\int_{t_i}^T\left|f(r,Y_r,Z_r)-f(t_1(r),Y_{t_1(r)},Z_{t_1(r)})\right|dr\right|\mathcal{G}_{t_i}\right)\right)^p\nonumber\\
&\leq&\mathbb{E}\left(\max_{0\le i\le n-1}\mathbb{E}\left(\left.\int_{0}^T\left|f(r,Y_r,Z_r)-f(t_1(r),Y_{t_1(r)},Z_{t_1(r)})\right|dr\right|\mathcal{G}_{t_i}\right)\right)^p\nonumber\\
&\leq&C\mathbb{E}\left(\int_{0}^T\left|f(r,Y_r,Z_r)-f(t_1(r),Y_{t_1(r)},Z_{t_1(r)})\right|dr\right)^p\nonumber\\
&\leq&CT^{p-1}\mathbb{E}\int_0^T\left|f(r,Y_r,Z_r)-f(t_1(r),Y_{t_1(r)},Z_{t_1(r)})\right|^pdr\le C T^p|\pi|^{\frac{p}{2}}.
\end{eqnarray}
This proves the desired bound for the first summand in \eqref{s-7-9-1}. 
For the second  summand, we can write
\begin{eqnarray}   \nonumber
&&\mathbb{E}\max\limits_{0\leq i\leq n-1}\left|\mathbb{E}\left(\left.G^\pi_{t_i}\right|\mathcal{G}_{t_i}\right)\right|^p \\
&=&\mathbb{E}\max\limits_{0\leq i\leq n-1}\left|\mathbb{E}\left(\left.\int_{t_i}^T[g(Y_{r})-g(Y_{t_{2}(r)})]d\overleftarrow{B}_r\right|\mathcal{G}_{t_i}\right)\right|^p\nonumber\\
&=&\mathbb{E}\max\limits_{0\leq i\leq n-1}\left|\mathbb{E}\left(\left.\int_{0}^T[g(Y_{r})-g(Y_{t_{2}(r)})]d\overleftarrow{B}_r\right|\mathcal{G}_{t_i}\right)-\int_{0}^{t_i}[g(Y_{r})-g(Y_{t_{2}(r)})]d\overleftarrow{B}_r\right|^p\nonumber\\
&\leq&2^{p-1}\left[\mathbb{E}\max\limits_{0\leq i\leq n-1}\left|\mathbb{E}\left(\left.\int_{0}^T[g(Y_{r})-g(Y_{t_{2}(r)})]d\overleftarrow{B}_r\right|\mathcal{G}_{t_i}\right)\right|^p\right.\nonumber\\
&&\qquad\quad\left.+\mathbb{E}\max\limits_{0\leq i\leq n-1}\left|\int_{0}^{t_i}[g(Y_{r})-g(Y_{t_{2}(r)})]d\overleftarrow{B}_r\right|^p\right]=: 2^{p-1} \left[ A_1 + A_2 \right]. \label{eq89}
\end{eqnarray}
Using Doob's maximal inequality,  we obtain
\begin{equation}
A_1 \le  c_p  \mathbb{E}\left|\int_{0}^T[g(Y_{r})-g(Y_{t_{2}(r)})]d\overleftarrow{B}_r\right|^p \label{eq90},
\end{equation}
and
\begin{eqnarray}  \nonumber
A_2 & \le&  \mathbb{E}\max\limits_{0\leq i\leq n-1}\left|\int_{0}^{T}[g(Y_{r})-g(Y_{t_{2}(r)})]d\overleftarrow{B}_r-\int_{t_i}^T[g(Y_{r})-g(Y_{t_{2}(r)})]d\overleftarrow{B}_r\right|^p\\
&\le& c'_p \mathbb{E}\left|\int_{0}^T[g(Y_{r})-g(Y_{t_{2}(r)})]d\overleftarrow{B}_r\right|^p. \label{eq91}
\end{eqnarray}
Substituting \eqref{eq90} and \eqref{eq91} into \eqref{eq89} and applying
Burkholder-Davis-Gunday's inequality, H\"{o}lder's inequality, $(iii)$ in Assumption  {\bf (A)}, Corollary \ref{l.3.7} and Theorem \ref{t.3.1} (b),  we obtain
\begin{eqnarray}
 \mathbb{E}\max\limits_{0\leq i\leq n-1}\left|\mathbb{E}\left(\left.G^\pi_{t_i}\right|\mathcal{G}_{t_i}\right)\right|^p 
&\leq& C\mathbb{E}\left(\int_{0}^T|g(Y_{r})-g(Y_{t_{2}(r)})|^2dr\right)^{\frac{p}{2}}\nonumber\\
&\leq& CT^{\frac{p-2}{2}}\mathbb{E}\int_0^T|Y_{r}-Y_{t_{2}(r)}|^pdr\leq CT^{\frac{p}{2}}|\pi|^{\frac{p}{2}}. \label{s-7-14-1}
\end{eqnarray}
Therefore, \eqref{s-7-13-1} and \eqref{s-7-14-1} yield the desired inequality.
\end{proof}

\bigskip

Finally, let us give the proof of our main result on the rate of convergence of the  implicit numerical scheme.
\begin{proof}[Proof of Theorem \ref{t-3-10}]
In this proof, let $C>0$ be a generic constant depending only on $T$, $p$ and all the constants appearing in the assumptions in this theorem, and not depending on  the partition $\pi$.

Note  that both $Y_{t_i}$ and $Y_{t_i}^\pi$ are $\mathcal{F}_{t_i}\subseteq\mathcal{G}_{t_i}$-measurable for all $i=n, n-1,\dots, 0$ and that $Z_{t}^{1,\pi}$ is $\mathcal{G}_t$-measurable for all $t\in[0,T]$. By \eqref{eq7} we obtain, for $i=n-1,\dots, 1, 0$,
\begin{eqnarray}\label{s-7-12}
Y_{t_i}-Y_{t_i}^\pi&=&\mathbb{E}\left(\xi-\xi^\pi\big|\mathcal{G}_{t_i}\right)+\mathbb{E}\left(\left.\sum_{j=i}^{n-1}[f(t_j,Y_{t_j},Z_{t_j})-f(t_j,Y_{t_j}^\pi,Z_{t_j}^\pi)]\Delta_j\right|\mathcal{G}_{t_i}\right)\nonumber\\
&&\!\!\!+\mathbb{E}\left(\left.\sum_{j=i}^{n-1}[g(Y_{t_{j+1}})-g(Y_{t_{j+1}}^\pi)]\Delta B_j\right|\mathcal{G}_{t_i}\right)+\mathbb{E}\left(R^\pi_{t_i}\big|\mathcal{G}_{t_i}\right)+\mathbb{E}\left(G^\pi_{t_i}\big|\mathcal{G}_{t_i}\right).
\end{eqnarray}
To simplify the notation we denote, for $i=n, n-1,\dots, 0$,
\begin{eqnarray*}
\delta Y^\pi_{t_i}&=&Y_{t_i}-Y^\pi_{t_i},\qquad      \delta Z^\pi_{t_i}=Z_{t_i}-Z^\pi_{t_i}, \\
 \tilde{f}_{t_i}^\pi &=& f(t_i,Y_{t_i},Z_{t_i})-f(t_i,Y_{t_i}^\pi,Z_{t_i}^\pi),\quad  \tilde{g}_{t_i}^\pi=g(Y_{t_i})-g(Y^\pi_{t_i}),
\end{eqnarray*}
and 
\[
\delta Z^{1,\pi}_t = Z_t-Z^{1,\pi}_t, 
\]
for any $t\in [0,T]$.
By convention, $\delta Y_{t_n}^\pi=\xi-\xi^\pi$ and $\delta Z_{t_n}^\pi=0$. Then, we can rewrite \eqref{s-7-12} as
\begin{eqnarray}
\delta Y_{t_i}^\pi&=&\mathbb{E}\left(\xi-\xi^\pi\big|\mathcal{G}_{t_i}\right)+\mathbb{E}\left(\left.\sum_{j=i}^{n-1}\tilde{f}_{t_j}^\pi\Delta_j\right|\mathcal{G}_{t_i}\right)+\mathbb{E}\left(\left.\sum_{j=i}^{n-1}\tilde{g}^\pi_{t_{j+1}}\Delta B_j\right|\mathcal{G}_{t_i}\right) \nonumber \\
&& +\mathbb{E}\left(R^\pi_{t_i}\big|\mathcal{G}_{t_i}\right)+\mathbb{E}\left(G^\pi_{t_i}\big|\mathcal{G}_{t_i}\right).\nonumber\\
\end{eqnarray}
Thus, for $k=n-1,\dots, 0$,
\begin{eqnarray}
\max\limits_{k\leq i\leq n} |\delta Y_{t_i}^\pi|& \leq &\max\limits_{0\leq i\leq n}\mathbb{E}\left(|\xi-\xi^\pi|\big|\mathcal{G}_{t_i}\right)+\max\limits_{k\leq i\leq n}\mathbb{E}\left(\left.\sum_{j=k}^{n-1}|\tilde{f}_{t_j}^\pi|\Delta_j\right|\mathcal{G}_{t_i}\right) \nonumber \\
&& +\max\limits_{k\leq i\leq n}\mathbb{E}\left(\left.\sum_{j=i}^{n-1}\tilde{g}^\pi_{t_{j+1}}\Delta B_j\right|\mathcal{G}_{t_i}\right)
+\max\limits_{0\leq i\leq n}\mathbb{E}\left(|R^\pi_{t_j}|\big|\mathcal{G}_{t_i}\right) \nonumber \\
&& +\max\limits_{0\leq i\leq n}\mathbb{E}\left(G^\pi_{t_i}\big|\mathcal{G}_{t_i}\right).
\end{eqnarray}
Then, by Doob's maximal inequality, the Lipschitz condition on $f$, \eqref{e.3.25} and Lemma \ref{l-7-1}, we are able to show the following estimate: 
\begin{eqnarray}\label{s-7-18}
&&\mathbb{E}\max\limits_{k\leq i\leq n} |\delta Y_{t_i}^\pi|^p\nonumber\\
&\leq& C\left[\mathbb{E}|\xi-\xi^\pi|^p+|\pi|^{\frac{p}{2}}+\mathbb{E}\left|\sum_{j=k}^{n-1}|\tilde{f}_{t_j}^\pi|\Delta_j\right|^p+\mathbb{E}\max\limits_{k\leq i\leq n}\left|\mathbb{E}\left(\left.\sum_{j=i}^{n-1}\tilde{g}^\pi_{t_{j+1}}\Delta B_j\right|\mathcal{G}_{t_i}\right)\right|^p\right]\nonumber\\
&\leq&C\left[\mathbb{E}|\xi-\xi^\pi|^p+|\pi|^{\frac{p}{2}}+\mathbb{E}\left|\sum_{j=k}^{n-1}\left[|\delta Y_{t_j}^\pi|+|\delta Z_{t_j}^\pi|\right]\Delta_j\right|^p+\mathbb{E}\max\limits_{k\leq i\leq n}\left|\mathbb{E}\left(\left.\sum_{j=i}^{n-1}\tilde{g}^\pi_{t_{j+1}}\Delta B_j\right|\mathcal{G}_{t_i}\right)\right|^p\right]\nonumber\\
&\leq&C\left[\mathbb{E}|\xi-\xi^\pi|^p+|\pi|^{\frac{p}{2}}\right]+C(T-t_k)^p\mathbb{E}\max\limits_{k\leq i\leq n} |\delta Y_{t_i}^\pi|^p \nonumber \\
&&+C\mathbb{E}\left(\sum_{j=k}^{n-1}\left|Z_{t_j}\Delta_j-\mathbb{E}\left(\left.\int_{t_j}^{t_{j+1}}Z_r^{1,\pi}dr\right|\mathcal{F}_{t_j}\right)\right|\right)^p +C\mathbb{E}\max\limits_{k\leq i\leq n}\left|\mathbb{E}\left(\left.\sum_{j=i}^{n-1}\tilde{g}^\pi_{t_{j+1}}\Delta B_j\right|\mathcal{G}_{t_i}\right)\right|^p.\nonumber \\
\end{eqnarray}

Next, we will estimate the last two terms on the right-hand side of the above inequality. From the integral representation \eqref{e.3.23-s}, we know that $Z_r^{1,\pi}$ is independent of $\mathcal{F}_{0,t_j}^B$ for all $r\in[t_j,t_{j+1}],\ j=n-1,\dots, 0$. Thus, it holds that
\[
\mathbb{E}\left(\left.\int_{t_j}^{t_{j+1}}Z_r^{1,\pi}dr\right|\mathcal{F}_{t_j}\right)=\mathbb{E}\left(\left.\int_{t_j}^{t_{j+1}}Z_r^{1,\pi}dr\right|\mathcal{F}_{t_j}\vee \mathcal{F}_{0,t_j}^B\right)=\mathbb{E}\left(\left.\int_{t_j}^{t_{j+1}}Z_r^{1,\pi}dr\right|\mathcal{G}_{t_j}\right).
\]
Then, we get
\begin{eqnarray}
&&\mathbb{E}\left(\sum_{j=k}^{n-1}\left|Z_{t_j}\Delta_j-\mathbb{E}\left(\left.\int_{t_j}^{t_{j+1}}Z_r^{1,\pi}dr\right|\mathcal{F}_{t_j}\right)\right|\right)^p=\mathbb{E}\left(\sum_{j=k}^{n-1}\left|Z_{t_j}\Delta_j-\mathbb{E}\left(\left.\int_{t_j}^{t_{j+1}}Z_r^{1,\pi}dr\right|\mathcal{G}_{t_j}\right)\right|\right)^p\nonumber\\
& & \le 2^{p-1}\left[\mathbb{E}\left(\sum_{j=k}^{n-1}\int_{t_j}^{t_{j+1}}\mathbb{E}\left(\left.|Z_{t_j}-Z_r|\right|\mathcal{G}_{t_j}\right)dr\right)^p+\mathbb{E}\left(\sum_{j=k}^{n-1}\int_{t_j}^{t_{j+1}}\mathbb{E}\left(\left.|Z_{r}-Z_r^{1,\pi}|\right|\mathcal{G}_{t_j}\right)dr\right)^p\right]\nonumber\\
& &=:2^{p-1}[I_1+I_2].   \label{eq16}
\end{eqnarray}
H\"{o}lder's and Jessen's inequalities and \eqref{e-z} yield
\begin{eqnarray}
I_1&\leq& \mathbb{E}\left(\sum_{j=k}^{n-1}\Delta_j^{\frac{p-1}{p}}\left(\int_{t_j}^{t_{j+1}}\left(\mathbb{E}\left(|Z_{t_j}-Z_r|\big|\mathcal{G}_{t_j}\right)\right)^pdr\right)^{\frac{1}{p}}\right)^p\nonumber\\
&\leq & (T-t_k)^{p-1}\mathbb{E}\sum_{j=k}^{n-1}\int_{t_j}^{t_{j+1}}\left(\mathbb{E}\left(|Z_{t_j}-Z_r|\big|\mathcal{G}_{t_j}\right)\right)^pdr\nonumber\\
&\leq&(T-t_k)^{p-1}\mathbb{E}\sum_{j=k}^{n-1}\int_{t_j}^{t_{j+1}}\mathbb{E}\left(|Z_{t_j}-Z_r|^p\big|\mathcal{G}_{t_j}\right)dr\nonumber\\
&=&(T-t_k)^{p-1}\mathbb{E}\sum_{j=k}^{n-1}\int_{t_j}^{t_{j+1}}|Z_{t_j}-Z_r|^pdr\leq C|\pi|^{\frac{p}{2}}.   \label{eq17}
\end{eqnarray}
Applying Theorem 1.1 in \cite{G} and H\"{o}lder's inequality, we obtain
\begin{eqnarray}  \nonumber
I_2&\leq& C\mathbb{E}\left(\sum_{j=k}^{n-1}\int_{t_j}^{t_{j+1}}|Z_r-Z_r^{1,\pi}|dr\right)^p
=C\mathbb{E}\left(\int_{t_k}^T|Z_r-Z_r^{1,\pi}|dr\right)^p \\
&& \leq C(T-t_k)^{\frac{p}{2}}\mathbb{E}\left(\int_{t_k}^T|Z_r-Z_r^{1,\pi}|^2dr\right)^{\frac{p}{2}}.\label{s-7-21}
\end{eqnarray}
To estimate the last term on the right-hand side of \eqref{s-7-18}, adopting the notation $t_2(r)$ introduced in the proof of Lemma \ref{l-7-1}, 
we can write
\begin{eqnarray}  \label{eq11}
&&\mathbb{E}\max\limits_{k\leq i\leq n}\left|\mathbb{E}\left(\sum_{j=i}^{n-1}\tilde{g}^\pi_{t_{j+1}}\Delta B_j\Big |\mathcal{G}_{t_i}\right)\right|^p \nonumber\\
&=& \mathbb{E}\max\limits_{k\leq i\leq n-1}\left|\mathbb{E}\left(\sum_{j=i}^{n-1}[g(Y_{t_{j+1}})-g(Y_{t_{j+1}}^\pi)]\Delta B_j\Big|\mathcal{G}_{t_i}\right)\right|^p\nonumber\\ 
&=&\mathbb{E}\max\limits_{k\leq i\leq n-1}\left|\mathbb{E}\left(\int_{t_i}^T[g(Y_{t_2(r)})-g(Y_{t_{2}(r)}^\pi)]d\overleftarrow{B}_r\Big| \mathcal{G}_{t_i}\right)\right|^p\nonumber\\ 
&=&\mathbb{E}\max\limits_{k\leq i\leq n-1}\left|\mathbb{E}\left( \int_{t_k}^T[g(Y_{t_2(r)})-g(Y_{t_{2}(r)}^\pi)]d\overleftarrow{B}_r\Big |\mathcal{G}_{t_i}\right)-\int_{t_k}^{t_i}[g(Y_{t_2(r)})-g(Y_{t_{2}(r)}^\pi)]d\overleftarrow{B}_r\right|^p\nonumber\\ 
&\leq&2^{p-1}\left[\mathbb{E}\max\limits_{k\leq i\leq n-1}\left|\mathbb{E}\left(\left.\int_{t_k}^T[g(Y_{t_2(r)})-g(Y_{t_{2}(r)}^\pi)]d\overleftarrow{B}_r\right|\mathcal{G}_{t_i}\right)\right|^p\right.\nonumber\\
&&\qquad\qquad\left.+\mathbb{E}\max\limits_{k\leq i\leq n-1}\left(\int_{t_k}^{t_i}[g(Y_{t_2(r)})-g(Y_{t_{2}(r)}^\pi)]d\overleftarrow{B}_r\right)^p\right] = 2^{p-1} \left[  B_1 +B_2\right]. 
\end{eqnarray}
Using Doob's maximal inequality, we obtain
\begin{equation} \label{eq8}
B_1 \le  c_p\mathbb{E}\left|\int_{t_k}^T[g(Y_{t_2(r)})-g(Y_{t_{2}(r)}^\pi)]d\overleftarrow{B}_r\right|^p
\end{equation}
and
 \begin{eqnarray}
B_2&=& \mathbb{E}\max\limits_{k\leq i\leq n-1}\left|\int_{t_k}^{T}[g(Y_{t_2(r)})-g(Y_{t_{2}(r)}^\pi)]d\overleftarrow{B}_r-\int_{t_i}^T[g(Y_{t_2(r)})-g(Y_{t_{2}(r)}^\pi)]d\overleftarrow{B}_r\right|^p  \nonumber \\
&\le& c'_p \mathbb{E}\left|\int_{t_k}^T[g(Y_{t_2(r)})-g(Y_{t_{2}(r)}^\pi)]d\overleftarrow{B}_r\right|^p.  \label{eq9}
\end{eqnarray} 
Substituting \eqref{eq8}  and \eqref{eq9} into \eqref{eq11} and
applying Burkholder-Davis-Gundy's inequality, H\"{o}lder's inequality and $(iii)$ in Assumption  {\bf (A)}, we obtain
\begin{eqnarray}
&&\mathbb{E}\max\limits_{k\leq i\leq n}\left|\mathbb{E}\left(\sum_{j=i}^{n-1}\tilde{g}^\pi_{t_{j+1}}\Delta B_j\Big |\mathcal{G}_{t_i}\right)\right|^p\\
&\leq& C\mathbb{E}\left(\int_{t_k}^T|g(Y_{t_2(r)})-g(Y_{t_2(r)}^\pi)|^2dr\right)^{\frac{p}{2}}
\leq C(T-t_k)^{{ \frac{p}{2}}}\mathbb{E}\max\limits_{k\leq i\leq n-1}|\delta Y^\pi_{t_{i+1}}|^p\nonumber\\
&\leq& C(T-t_k)^{{\frac{p}{2}}}\mathbb{E}\max\limits_{k\leq i\leq n}|\delta Y^\pi_{t_{i}}|^p.  \label{s-7-22}
\end{eqnarray}
From \eqref{s-7-18}, \eqref{eq16}, \eqref{eq17}, \eqref{s-7-21} and \eqref{s-7-22}, it follows 
\begin{eqnarray}   \nonumber
\mathbb{E}\max\limits_{k\leq i\leq n} |\delta Y_{t_i}^\pi|^p  &\leq& C\left[\mathbb{E}|\xi-\xi^\pi|^p+|\pi|^{\frac{p}{2}}\right]+C(T-t_k)^{{ \frac{p}{2}}}\mathbb{E}\max\limits_{k\leq i\leq n} |\delta Y_{t_i}^\pi|^p  \\
&& +C(T-t_k)^{\frac{p}{2}}\mathbb{E}\left(\int_{t_k}^T|Z_r-Z_r^{1,\pi}|^2dr\right)^{\frac{p}{2}}. \label{s-7-23}
\end{eqnarray}
By  Burkholder-Davis-Gundy's inequality, we have
\begin{equation}\label{s-7-24}
\mathbb{E} \left( \int_{t_k}^{T}|Z_r- Z_r^{1,\pi} |^2 dr
\right)^{\frac{p}{2}}\le c_p \mathbb{E} \left\vert \int_{t_k}^T
(Z_r- Z_r^{1,\pi}) dW_r\right\vert^p\,.
\end{equation}
From \eqref{eq7}, we obtain
\begin{equation}\label{s-7-25}
\int_{t_k}^T (Z_r-Z_r^{1,\pi}) dW_r =-\delta Y_{t_k}^\pi+ \xi-\xi^\pi +\sum_{i=k}^{n-1}
\widetilde{f}_{t_i}^\pi {\Delta}_{i}+\sum_{i=k}^{n-1}\tilde{g}^\pi_{t_{i+1}}\Delta B_i+R_{t_k}^\pi+G_{t_k}^\pi.
\end{equation}
Then, it follows from \eqref{s-7-24}, \eqref{s-7-25},  Lemma \ref{l-7-1} and the arguments used in the proof of  \eqref{s-7-23},   that there exists a constant $C_1>0$ independent of  the partition $\pi$ such that
\begin{eqnarray*}
\mathbb{E} \left( \int_{t_k}^{T}|Z_r- Z_r^{1,\pi} |^2 dr
\right)^{\frac{p}{2}}
&\leq& C_1\left[\mathbb{E}|\xi-\xi^\pi|^p+|\pi|^{\frac{p}{2}}\right]+C_1(T-t_k)^{{ \frac{p}{2}}}\mathbb{E}\max\limits_{k\leq i\leq n} |\delta Y_{t_i}^\pi|^p\nonumber\\
&&+C_1(T-t_k)^{\frac{p}{2}}\mathbb{E}\left(\int_{t_k}^T|Z_r-Z_r^{1,\pi}|^2dr\right)^{\frac{p}{2}}.
\end{eqnarray*}
If $C_1(T-t_k)^{\frac{p}{2}}<\frac{1}{2}$, then we get
\begin{equation}\label{s-7-26}
\mathbb{E} \left( \int_{t_k}^{T}|Z_r- Z_r^{1,\pi} |^2 dr
\right)^{\frac{p}{2}}\leq 2C_1\left[\mathbb{E}|\xi-\xi^\pi|^p+|\pi|^{\frac{p}{2}}\right]+2C_1(T-t_k)^{{ \frac{p}{2}}}\mathbb{E}\max\limits_{k\leq i\leq n} |\delta Y_{t_i}^\pi|^p.
\end{equation}
Substituting \eqref{s-7-26} into \eqref{s-7-23}, we can find a constant $C_2>0$ independent of the partition $\pi$ such that
\begin{equation}\label{s-7-27}
\mathbb{E}\max\limits_{k\leq i\leq n} |\delta Y_{t_i}^\pi|^p\leq C_2\left[\mathbb{E}|\xi-\xi^\pi|^p+|\pi|^{\frac{p}{2}}\right]+C_2(T-t_k)^{{ \frac{p}{2}}}\mathbb{E}\max\limits_{k\leq i\leq n} |\delta Y_{t_i}^\pi|^p.
\end{equation}
Fix a positive constant $\delta$ independent of the partition $\pi$ such that
\begin{eqnarray}
C_1(3\delta)^{\frac{p}{2}}<\frac{1}{2}, \qquad 
C_2(3\delta)^{{ \frac{p}{2}}}<\frac{1}{2}, \qquad 
2\delta <T.\nonumber
\end{eqnarray}
Denote $l=\left[\frac{T}{2\delta}\right]$. Then $l\geq 1$ is an integer independent of the partition $\pi$. If $|\pi|<\delta$, then for the partition $\pi$ we choose $n-1>i_1>\cdots >i_l \geq 0$ such that $T-2\delta\in(t_{i_1-1},t_{i_1}]$, $T-4\delta\in(t_{i_2-1},t_{i_2}]$, $\dots$, $T-2\delta l\in[0,t_{i_l}]$ (with $t_{-1}=0$). For simplicity, we denote $t_{i_0}=T$ and $t_{i_{l+1}}=0$. Each interval $[t_{i_{j+1}},t_{i_j}]$, $j=0, 1, \dots, l$, has length less than $3\delta$, that is, $|t_{i_j}-t_{i_{j+1}}|<3\delta$.

On each interval $[t_{i_{j+1}},t_{i_j}]$, $j=0, 1, \dots, l$, we carry out the same analysis as that in \eqref{s-7-18}-\eqref{s-7-27}, and in this way we obtain
\begin{eqnarray}  
\mathbb{E} \left( \int_{t_{i_{j+1}}}^{t_{i_j}}|Z_r- Z_r^{1,\pi} |^2 dr
\right)^{\frac{p}{2}} \le 2C_1\mathbb{E}|\delta Y_{t_{i_j}}^\pi|^p+2C_1|\pi|^{\frac{p}{2}}  +2C_1(t_{i_j}-t_{i_{j+1}})^{{ \frac{p}{2}}}\mathbb{E}\max\limits_{i_{j+1}\leq i\leq i_j}|\delta Y_{t_i}^\pi|^p \nonumber\\
\label{s-7-28}
\end{eqnarray}
and
\begin{eqnarray*}
\mathbb{E}\max\limits_{i_{j+1}\leq i\leq i_j}|\delta Y_{t_i}^\pi|^p&\leq& C_2\left[\mathbb{E}|\delta Y_{t_{i_j}}^\pi|^p+|\pi|^{\frac{p}{2}}\right]+C_2(t_{i_j}-t_{i_{j+1}})^{{ \frac{p}{2}}}\mathbb{E}\max\limits_{i_{j+1}\leq i\leq i_j}|\delta Y_{t_i}^\pi|^p\nonumber\\
&\leq&C_2\left[\mathbb{E}|\delta Y_{t_{i_j}}^\pi|^p+|\pi|^{\frac{p}{2}}\right]+\frac{1}{2}\mathbb{E}\max\limits_{i_{j+1}\leq i\leq i_j}|\delta Y_{t_i}^\pi|^p.
\end{eqnarray*}
Hence,
\begin{eqnarray*}
\mathbb{E}\max\limits_{i_{j+1}\leq i\leq i_j}|\delta Y_{t_i}^\pi|^p\leq 2C_2\left[\mathbb{E}|\delta Y_{t_{i_j}}^\pi|^p+|\pi|^{\frac{p}{2}}\right].
\end{eqnarray*}
By recurrence, we have
\begin{eqnarray*}
\mathbb{E}\max\limits_{i_{j+1}\leq i\leq i_j}|\delta Y_{t_i}^\pi|^p&\leq&(2C_2)^{j+1}\mathbb{E}|\xi-\xi^\pi|^p+2C_2[1+2C_2+\cdots+(2C_2)^j]|\pi|^{\frac{p}{2}}\nonumber\\
&\leq& (2C_2)^{l+1}\mathbb{E}|\xi-\xi^\pi|^p+\frac{2C_2(1-(2C_2)^{l+1})}{1-2C_2}|\pi|^{\frac{p}{2}}.
\end{eqnarray*}
Thus, taking $K_1=(l+1)^{p}\max\left\{(2C_2)^{l+1},\frac{2C_2(1-(2C_2)^{l+1})}{1-2C_2}\right\}$, we have the following estimate
\begin{equation}\label{s-7-29}
\mathbb{E}\max\limits_{0\leq i\leq n}|\delta Y_{t_i}^\pi|^p\leq (l+1)^{p-1}\sum_{j=0}^{l}\mathbb{E}\max\limits_{i_{j+1}\leq i\leq i_j}|\delta Y_{t_i}^\pi|^p\leq K_1\left[\mathbb{E}|\xi-\xi^\pi|^p+|\pi|^{\frac{p}{2}}\right].
\end{equation}
Plugging \eqref{s-7-29} into \eqref{s-7-28}, we obtain
\begin{equation*}
\mathbb{E} \left( \int_{t_{i_{j+1}}}^{t_{i_j}}|Z_r- Z_r^{1,\pi} |^2 dr
\right)^{\frac{p}{2}}\leq 2C_1|\pi|^{\frac{p}{2}}+2C_1K_1(1+T^{{ \frac{p}{2}}})\left[\mathbb{E}|\xi-\xi^\pi|^p+|\pi|^{\frac{p}{2}}\right].
\end{equation*}
Then, by taking $K_2=(l+1)^{\frac{p}{2}}(2C_1+2C_1K_1(1+T^{{ \frac{p}{2}}}))$, we have
\begin{eqnarray}\label{s-7-30}
\mathbb{E} \left( \int_{0}^{T}|Z_r- Z_r^{1,\pi} |^2 dr
\right)^{\frac{p}{2}}&\leq& (l+1)^{\frac{p-2}{2}}\sum_{i=0}^{l}\mathbb{E} \left( \int_{t_{i_{j+1}}}^{t_{i_j}}|Z_r- Z_r^{1,\pi} |^2 dr
\right)^{\frac{p}{2}}\nonumber\\
&\leq&(l+1)^{\frac{p}{2}}\left(2C_1|\pi|^{\frac{p}{2}}+2C_1K_1(1+T^{{ \frac{p}{2}}})\left[\mathbb{E}|\xi-\xi^\pi|^p+|\pi|^{\frac{p}{2}}\right]\right)\nonumber\\
&\leq&K_2\left[\mathbb{E}|\xi-\xi^\pi|^p+|\pi|^{\frac{p}{2}}\right].
\end{eqnarray}
Therefore, by taking $K=K_1+K_2$ and adding \eqref{s-7-29} and \eqref{s-7-30}, we deduce \eqref{s-3-34-1}.
\end{proof}

\setcounter{equation}{0}

\end{document}